\documentclass[10pt]{article}
\usepackage{amsmath}
\usepackage{amssymb}
\usepackage{amsthm}
\usepackage{amsfonts}
\usepackage{algorithm}
\usepackage{algorithmic}
\usepackage{tikz}
\usepackage{graphicx}
\usepackage{float}
\usepackage{framed}
\usepackage[hang,flushmargin]{footmisc}

 \newtheorem{theorem}{Theorem}[section]
 
 \newtheorem{lemma}[theorem]{Lemma}
 \newtheorem{proposition}[theorem]{Proposition}
 \theoremstyle{definition}
 \newtheorem{definition}[theorem]{Definition}
 \theoremstyle{remark}
 \newtheorem{remark}[theorem]{Remark}
 \newtheorem*{example}{Example}
 \newtheorem{alg}{Algorithm}

\DeclareMathOperator{\ord}{ord} \DeclareMathOperator{\Card}{Card}
 \DeclareMathOperator{\rk}{rk}
\DeclareMathOperator{\trdeg}{tr.deg}
\DeclareMathOperator{\Ker}{Ker}
\DeclareMathOperator{\qf}{qf} \DeclareMathOperator{\Eord}{Eord}
\DeclareMathOperator{\lex}{lex}
\begin{document}
\def\D{\displaystyle}

\begin{center}
\large {\bf New Multivariate Dimension Polynomials of Inversive Difference Field Extensions}
\normalsize
\smallskip

Alexander  Levin

The Catholic University of America

Washington, D. C.  20064, USA

levin@cua.edu   

https://sites.google.com/a/cua.edu/levin
\end{center}

\begin{abstract}
We introduce a new type of reduction of inversive difference polynomials that is associated with a partition of the basic set of automorphisms $\sigma$ and uses a generalization of the concept of effective order of a difference polynomial. Then we develop the corresponding method of characteristic sets and apply it to prove the existence and obtain a method of computation of multivariate dimension polynomials of a new type that describe
the transcendence degrees of intermediate fields of finitely generated inversive difference field extensions obtained by adjoining transforms of the generators whose orders with respect to the components of the partition of $\sigma$ are bounded by two sequences of natural numbers. We show that such dimension polynomials carry essentially more invariants (that is, characteristics of the extension that do not depend on the set of its difference generators) than standard (univariate) difference dimension polynomials. We also show how the obtained results can be applied to the equivalence problem for systems of algebraic difference equations.
\end{abstract}

{\bf Key words:} difference polynomial, dimension polynomial, reduction,\\ effective order, characteristic set.

\section{Introduction}

This paper is dedicated to the memory of my dear teacher, Alexander Vasilyevich Mikhalev, who has made profound contributions to several areas of mathematics, especially to various branches of algebra including the ring theory, homological algebra, differential and difference algebra, computer algebra, algebraic $K$-theory, topological algebra, and coding theory. In his  works on differential and difference algebraic structures \cite{KLMP0}, \cite{KLMP}, \cite{MP4}, \cite{LM1} - \cite{LM7}, \cite{MP1} - \cite{MP4} and in some other papers A. V. Mikhalev obtained a number of fundamental results on differential and difference rings and modules, characteristic sets of differential and difference polynomials, and computational analysis of systems of algebraic differential and difference equations.  He has also presented excellent expositions of ideas and methods of differential and difference algebra in his books \cite{KLMP}, \cite{MP4} and papers \cite{MP1} and \cite{MP3}. Of special note is the paper \cite{MP2} where A. V. Mikhalev and E. V. Pankratiev discovered a very interesting relationship between E. Kolchin's differential dimension polynomials and A. Einstein's concept of strength of a system of algebraic differential equations. Actually, the authors showed that the strength of such a system in the
sense of A. Einstein is expressed by certain differential dimension polynomial associated with the system. (The concept of a differential dimension polynomial was introduced in \cite{Kolchin1}; many properties of such polynomials can be found in \cite{Kolchin2}.) Furthermore, they showed how the algebraic technique for computing differential dimension polynomials can be applied to the computation of the strength of fundamental systems of differential equations of mathematical physics. A similar interpretation of difference dimension polynomials and examples of computation of the strength of systems of algebraic difference equations can be found in \cite[Section[6.4]{KLMP}, \cite{L2} and \cite[Section 7.7]{L4}.

In addition to the fact that a difference dimension polynomial associated with a system of algebraic difference equations expresses the strength of such a system in the sense of A. Einstein (the significant role of this characteristic in the theory of equations of mathematical physics is described in \cite{E}), the important role of difference dimension polynomials is determined by at least three more factors. First, a difference  dimension polynomial of a finitely generated difference field extension (or of a system of algebraic difference equations that defines such an extension) carries certain invariants, i.e., characteristics of the extension that do not  change when we switch to another system of difference generators (with the corresponding change of the defining equations), see, for example, \cite[Chapter 6]{KLMP} and \cite[Chapter 4]{L4}. In this connection, one should mention the results on multivariate difference dimension polynomials associated with partitions of the basic set of translations, see \cite{L2}, \cite{L3}, \cite{L6}, and \cite[Chapter 3]{L4}. It turned out that they carry more such invariants than their univariate counterparts. (See also \cite{L8} where the results on multivariate difference dimension polynomials are generalized to the difference-differential case.) Second, properties of difference dimension polynomials associated with prime difference polynomial ideals provide a powerful tool in the dimension theory of difference algebras, see \cite[Chapter 7]{KLMP}, \cite[Section 4.6]{L4}, and \cite{L7}. Finally, the results on difference dimension polynomials can be naturally extended to algebraic and differential algebraic structures with a finitely generated commutative group action, see \cite{L5}, \cite{LM1}, and \cite{LM3}.

In this paper we introduce a reduction of inversive difference polynomials associated with a fixed partition of the set of basic translations. This reduction takes into account the effective orders of inversive difference polynomials with respect to the elements of the partition (we generalize the concept of the effective order of an ordinary difference polynomial defined in \cite[Chapter 2, Section 4]{Cohn}). Note that the idea of using a generalized effective order for (non-inversive) difference polynomials to obtain bivariate difference dimension polynomials of a new type was first explored in \cite{L9}. We consider a new type of characteristic sets that are associated with the introduced reduction and use their properties to prove the existence of a multivariate dimension polynomial of a finitely generated inversive difference field extension that describes the transcendence degrees of intermediate fields obtained by adjoining transforms of the generators whose orders with respect to the elements of the given partitions lie between two given natural numbers. This dimension polynomial is a polynomial in $2p$ variables where $p$ is the number of subsets in the partition of the basic set of translations. We determine invariants of such polynomials, that is, numerical characteristics of the extension that are carried by any its dimension polynomial and that do not depend on the system of difference generators the polynomial is associated with. Furthermore, we show that the introduced multivariate dimension polynomials carry essentially more invariants of the corresponding inversive difference field extensions than the univariate dimension polynomials of inversive difference modules and field extensions introduced in \cite{L1}. Note that while the study of difference algebraic structures deals with their endomorphisms and power products of basic translations with nonnegative exponents, inversive difference rings, fields and modules are considered together with the free commutative group generated by a set of basic automorphisms. Therefore, while the dimension theory of difference rings and modules is close to its differential counterpart, the study of inversive difference algebraic structures (including the study of dimensional characteristics of such structures) encounters many problems caused by the fact that one has to consider negative powers of basic translations.

\section{Preliminaries}

Throughout the paper, $\mathbb{N}$, $\mathbb{Z}$, $\mathbb{Z}_{\leq 0}$, and $\mathbb{Q}$ denote the sets of all non-negative integers, non-positive integers, integers, and rational numbers, respectively. If $S$ is a finite set, then $\Card S$ denotes the number of elements of $S$. For any positive integer $m$, $\leq_{P}$ will denote the product order on $\mathbb{N}^{m}$, that is, a partial order such that $(a_{1},\dots, a_{m})\leq_{P}(a'_{1},\dots, a'_{m})$ if and only if $a_{i}\leq a'_{i}$ for $i=1,\dots, m$. The lexicographic order will be denoted by $\leq_{\lex}$.

By a ring we always mean an associative ring with unity. Every ring homomorphism is unitary (maps unity to unity), every subring of a ring contains the unity of the ring, and every algebra over a commutative ring is
unitary. Every field considered in this paper is supposed to have zero characteristic. $\mathbb{Q}[t_{1},\dots, t_{p}]$ will denote the ring of polynomials in variables $t_{1},\dots, t_{p}$ over $\mathbb{Q}$.

By a {\em difference ring} we mean a commutative ring $R$ considered together with a finite set $\sigma = \{\alpha_{1},\dots, \alpha_{m}\}$ of mutually commuting injective endomorphisms of $R$ called {\em translations}.  The set $\sigma$ is called the {\em basic set} of the difference ring $R$, which is also called a $\sigma$-ring. If $R$ is a field, it is called a {\em difference field} or a $\sigma$-field. (We will often use prefix $\sigma$- instead of the adjective ''difference''.)

If all translations of $R$ are automorphisms, we set $\sigma^{\ast} = \{\alpha_{1},\dots, \alpha_{m}, \alpha^{-1}_{1},\dots,$ $ \alpha^{-1}_{m}\}$ and say that $R$ is an {\em inversive difference ring} or a $\sigma^{\ast}$-ring.  If a difference (respectively, inversive difference) ring $R$ is a field, it is called a {\em difference} (or $\sigma$-) {field} (respectively, an {\em inversive difference} (or $\sigma^{\ast}$-) {field}).

If $R$ is an inversive difference ring with a basic set $\sigma = \{\alpha_{1},\dots, \alpha_{m}\}$, then $\Gamma$ will denote the free commutative group of all power products of the form $\gamma = \alpha_{1}^{k_{1}}\dots \alpha_{m}^{k_{m}}$ where $k_{i}\in\mathbb{Z}$ ($1\leq i\leq m$). The {\em order} of such an element $\gamma$  is defined as $\ord\,\gamma = \sum_{i=1}^{m}|k_{i}|$; furthermore, for every $r\in\mathbb{N}$, we set $\Gamma(r) = \{\gamma\in \Gamma\,|\, \ord\,\gamma\leq r\}$.

A subring (ideal) $R_{0}$ of a $\sigma$-ring $R$ is said to be a difference  (or $\sigma$-) subring of $R$ (respectively, difference (or $\sigma$-) ideal of $R$) if $R_{0}$ is closed with respect to the action of any translation $\alpha_{i}\in\sigma$. A $\sigma$-ideal $I$ of a $\sigma$-ring $R$ is called {\em reflexive} if the inclusion $\alpha_{i}(a)\in I$ ($a\in R,\, \alpha_{i}\in \sigma$) implies the inclusion $a\in I$. (If $R$ is an inversive difference ($\sigma^{\ast}$-) ring, this property means that $I$ is closed with respect to every automorphism from the set $\sigma^{\ast}$). If a prime ideal $P$ of $R$ is closed with respect to the action of any $\alpha_{i}\in\sigma$, it is called a {\em prime difference} (or $\sigma$-) {\em ideal} of $R$. If $R$ is an inversive difference ring and a prime $\sigma$-ideal is reflexive, it is referred to as a prime $\sigma^{\ast}$-ideal of $R$.

If $R$ is a $\sigma$-ring and $S\subseteq R$, then the intersection $I$ of all $\sigma$-ideals of $R$ containing the set $S$ is the smallest $\sigma$-ideal of $R$ containing $S$; it is denoted by $[S]$. If the set $S$ is finite, $S = \{a_{1},\dots, a_{r}\}$, we say that the $\sigma$-ideal $I$ is finitely generated (we write this as $I = [a_{1},\dots, a_{r}]$) and call $a_{1},\dots, a_{r}$ difference (or $\sigma$-) generators of $I$. If the $\sigma$-ring $R$ is inversive, then the smallest $\sigma^{\ast}$-ideal of $R$ containing a subset $S$ of $R$ is denoted by $[S]^{\ast}$. Elements of the set $S$ are called $\sigma^{\ast}$-generators of this ideal; if $S = \{a_{1},\dots, a_{r}\}$, we write $[a_{1},\dots, a_{r}]^{\ast}$ and say that the $\sigma^{\ast}$-ideal is finitely generated and call $a_{1},\dots, a_{r}$ its $\sigma^{\ast}$-generators. Clearly, $[S]^{\ast}$ is generated, as an ideal, by the set $\{\gamma(a)\,|\, a\in S,\, \gamma\in\Gamma\}$. (In what follows we will often write $\gamma a$ instead of $\gamma(a)$.)

If $R$ is a $\sigma^{\ast}$-ring, then an expression of the form $\sum_{\gamma\in\Gamma}a_{\gamma}\gamma$, where $a_{\gamma}\in R$ for any $\gamma\in\Gamma$ and only finitely many elements $a_{\gamma}$ are different from $0$, is called a {\em $\sigma^{\ast}$-operator} over $R$. It is an endomorphism of the additive group of $R$; if $C = \sum_{\gamma\in\Gamma}a_{\gamma}\gamma$ and $f\in R$, then $C(f) = \sum_{\gamma\in\Gamma}a_{\gamma}\gamma(f)$.
Two $\sigma^{\ast}$-operators $\sum_{\gamma\in\Gamma}a_{\gamma}\gamma$ and $\sum_{\gamma\in\Gamma}b_{\gamma}\gamma$ are considered to be equal if and only if $a_{\gamma} = b_{\gamma}$ for any $\gamma\in\Gamma$. The set of all $\sigma^{\ast}$-operators over $R$ will be denoted by $\mathcal{E}_{R}$. This set, which has a natural structure of an $R$-module generated by $\Gamma$, becomes a ring if one sets $\gamma a = \gamma(a)\gamma$ for any $a\in R$, $\gamma\in\Gamma$ and extends this rule to the multiplication of any two $\sigma^{\ast}$-operators by distributivity. The resulting ring $\mathcal{E}_{R}$ is called the ring of $\sigma^{\ast}$-operators over $R$. Clearly, if $I$ is a $\sigma^{\ast}$-ideal of $R$, $I = [f_{1},\dots, f_{k}]^{\ast}$, then every element of $I$ is of the form $\sum_{i=1}^{q}C_{i}(f_{i})$ ($q\in\mathbb{N}$) where $C_{1},\dots, C_{q}\in\mathcal{E}_{R}$.

If $L$ is a difference ($\sigma$-) field and its subfield $K$ is also a $\sigma$-subring of $L$, then  $K$ is said to be a difference (or $\sigma$-) subfield of $L$; $L$, in turn, is called a difference (or $\sigma$-) field extension or a $\sigma$-overfield of $K$. In this case we also say that we have a $\sigma$-field extension $L/K$. If the $\sigma$-field $L$ is inversive and $K$ is a $\sigma$-subfield of $L$ such that $\alpha(K)\subseteq K$ for any $\alpha\in \sigma^{\ast}$, we say that $K$ is an inversive difference (or $\sigma^{\ast}$-) subfield of $L$ or that we have a $\sigma^{\ast}$-field extension $L/K$. In the last case, if $S\subseteq K$, then the smallest $\sigma^{\ast}$-subfield of $L$ containing $K$ and $S$ is denoted by $K\langle S \rangle^{\ast}$. $S$ is said to be the set of {\em $\sigma^{\ast}$-generators} of $K\langle S \rangle^{\ast}$ over $K$. If the set $S$ is finite, $S = \{\eta_{1},\dots,\eta_{n}\}$, we say that $L/K$ is a finitely generated inversive difference (or $\sigma^{\ast}$-) field extension. As a field, $L\langle S \rangle^{\ast} = K(\gamma a\,|\,\gamma\in\Gamma, \,a\in S)$.

Let $R$ and $R'$ be two difference rings with the same basic set $\sigma$, so that elements of $\sigma$ act on each of the rings as pairwise commuting endomorphisms. (More rigorously, we assume that there exist injective mappings of $\sigma$ into the sets of endomorphisms of the rings $R$ and $R'$ such that the images of any two elements of $\sigma$ commute. For convenience we will denote these images by the same symbols). A ring homomorphism $\phi: R \longrightarrow R'$ is called a {\em difference} (or $\sigma$-) {\em homomorphism} if $\phi(\alpha a) = \alpha\phi(a)$ for any $\alpha\in \sigma$, $a\in R$. It is easy to see that the kernel of such a mapping is a reflexive difference ideal of $R$.

In what follows we deal with inversive difference ($\sigma^{\ast}$-) rings and fields. If $R$ is such a ring and $Y =\{y_{1},\dots, y_{n}\}$ is a finite set of symbols, we can consider the polynomial ring $R[\Gamma Y]$, where $\Gamma Y$ denotes the set of symbols $\{\gamma y_{j}|\gamma\in\Gamma,  1\leq j\leq n\}$, as an inversive difference ring containing $R$ as its $\sigma^{\ast}$-subring. The corresponding inversive difference ring extension is defined by setting $\alpha(\gamma y_{j}) = (\alpha\gamma)y_{j}$ for any $\alpha\in\sigma^{\ast}$, $\gamma\in\Gamma$, $1\leq j\leq n$; it is denoted by $R\{y_{1},\dots, y_{n}\}^{\ast}$ and called the ring of inversive difference (or $\sigma^{\ast}$-) polynomials in $\sigma$-indeterminates $y_{1},\dots, y_{n}$ over $R$. A $\sigma^{\ast}$-ideal of $R\{y_{1},\dots, y_{n}\}^{\ast}$ is called {\em linear} if it is generated (as a $\sigma^{\ast}$-ideal) by homogeneous linear $\sigma^{\ast}$-polynomials, that is, $\sigma^{\ast}$-polynomials of the form $\sum_{i=1}^{d}a_{i}\gamma_{i}y_{k_{i}}$ ($a_{i}\in R$, $\gamma_{i}\in \Gamma$, $1\leq k_{i}\leq n$ for $i=1,\dots, d$). It is shown in \cite[Proposition 2.4.9]{L4} that if $R$ is a $\sigma^{\ast}$-field, then a linear $\sigma^{\ast}$-ideal of $R\{y_{1},\dots, y_{n}\}^{\ast}$ is prime.

If $K$ is an inversive difference ($\sigma^{\ast}$-) field, $f\in K\{y_{1},\dots, y_{n}\}^{\ast}$ and $\eta = (\eta_{1},\dots, \eta_{n})$ is an $n$-dimensional vector with coordinates in some $\sigma^{\ast}$-overfield of $K$, then $f(\eta)$ (or $f(\eta_{1},\dots, \eta_{n})$\,) denotes the result of the replacement of every entry $\gamma y_{i}$ in $f$ with $\gamma\eta_{i}$ ($\gamma\in\Gamma$, $1\leq i\leq n$).

If $\pi:R=K\{y_{1},\dots, y_{n}\}^{\ast}\rightarrow L=K\langle \eta_{1},\dots,\eta_{n}\rangle^{\ast}$ is a natural $\sigma$-homomorphism ($\pi(a) = a$ for any $a\in K$ and $y_{i}\mapsto \eta_{i}$), then $P = \Ker\pi$ is a prime $\sigma^{\ast}$-ideal of $R$ called the {\em defining ideal} of the extension $L/K$. In this case, $L$ is isomorphic to the $\sigma$-field $\qf(R/P)$, the quotient field of $R/P$ ($\eta_{i}\leftrightarrow y_{i}+P$).

Let $K$ be a $\sigma^{\ast}$-field and $\mathcal{U}$ a family of elements of some $\sigma^{\ast}$-overfield of $K$. We say that the family $\mathcal{U}$ is $\sigma$-{\em algebraically dependent} over $K$, if the family
$\Gamma\mathcal{U} = \{\gamma(u)\,\mid\,\gamma\in\Gamma,\, u\in \mathcal{U}\}$ is algebraically dependent over $K$ (that is, there exist elements $u_{1},\dots, u_{k}\in \Gamma\mathcal{U}$ and a nonzero polynomial $f$ in $k$ variables with coefficients in $K$ such that $f(u_{1},\dots, u_{k}) = 0$). Otherwise, the family $\mathcal{U}$ is said to be $\sigma$-{\em algebraically independent} over $K$.

If $L$ is a $\sigma^{\ast}$-overfield of a $\sigma^{\ast}$-field $K$, then a set $B\subseteq L$ is said to be a $\sigma$-{\em transcendence basis} of $L$ over $K$ if $B$ is $\sigma$-algebraically independent over
$K$ and every element $a\in L$ is $\sigma$-algebraic over $K\langle B\rangle$ (it means that the set $\{\gamma a\,\mid\,\tau\in\Gamma\}$ is algebraically dependent over the field $K\langle B\rangle^{\ast}$). If $L$ is a finitely generated $\sigma^{\ast}$-field extension of $K$, then all $\sigma$-transcendence bases of $L$ over $K$ are finite and have the same number of elements (see \cite[Proposition 4.1.6]{L4}). This number is called the $\sigma$-{\em transcendence degree} of $L$ over $K$ (or the $\sigma$-transcendence degree of the extension $L/K$); it is denoted by $\sigma$-$\trdeg_{K}L$.

The following theorem, whose prove can be found in \cite[Section 6.4]{KLMP}, introduces the (univariate) dimension polynomial of a finitely generated inversive difference field extension.

\begin{theorem}  Let $K$ be an inversive difference field with a basic set $\sigma = \{\alpha_{1},\dots, \alpha_{m}\}$ and $L = K\langle \eta_{1},\dots,\eta_{n}\rangle^{\ast}$ be a $\sigma^{\ast}$-field extension of $K$  generated by a finite set $\eta =\{\eta_{1},\dots,\eta_{n}\}$. Then there exists a polynomial $\phi_{\eta|K}(t)\in\mathbb{Q}[t]$ such that

{\em (i)}\, $\phi_{\eta|K}(r) = \trdeg_{K}K(\{\gamma\eta_{j} | \gamma\in\Gamma(r), 1\leq j\leq n\})$ for all sufficiently large $r\in\mathbb{N}$;

{\em (ii)}\, $\deg \phi_{\eta|K} \leq m$ and $\phi_{\eta|K}(t)$ can be written as \, $\phi_{\eta|K}(t) = \D\sum_{i=0}^{m}a_{i}{t+i\choose i}$ where $a_{0},\dots, a_{m}\in\mathbb{Z}$ and $2^{m}|a_{m}$\,\,.

{\em (iii)}\, $d =  \deg \phi_{\eta|K}$,\, $a_{m}$ and $a_{d}$ do not depend on the set of $\sigma^{\ast}$-generators $\eta$ of $L/K$ ($a_{d}\neq a_{m}$ if and only if $d < m$). Moreover,
$\D\frac{a_{m}}{2^{m}} = \sigma$-$\trdeg_{K}L$.

{\em (iv)}\, If the elements $\eta_{1},\dots, \eta_{n}$ are $\sigma$-algebraically independent over $K$, then
$$\phi_{\eta | K}(t) = n\D\sum_{k=0}^{m}(-1)^{m-k}2^{k}{n\choose k}{{t+k}\choose k}\,.$$
\end{theorem}

The polynomial $\phi_{\eta|K}(t)$ is called the {\em $\sigma^{\ast}$-dimension polynomial} of the $\sigma^{\ast}$-field extension $L/K$ associated with the system of $\sigma^{\ast}$-generators $\eta$. Methods and algorithms for computation of such polynomials can be found in \cite{KLMP}.

\bigskip

\begin{center}

DIMENSION POLYNOMIALS OF SUBSETS OF $\mathbb{Z}^{m}$

\end{center}

\bigskip

In what follows we present some results about numerical polynomials associated with subsets of $\mathbb{Z}^{m}$ ($m$ is a positive integer). The proofs of the corresponding statements can be found in
\cite{KLMP0} and \cite[Chapter 2]{KLMP}.

\begin{definition}
A polynomial in $p$ variables $f(t_{1}, \dots, t_{p})\in\mathbb{Q}[t_{1},\dots,t_{p}]$ is called {\bf numerical} if $f(r_{1},\dots, r_{p})\in\mathbb{Z}$ for all sufficiently large $(r_{1}, \dots, r_{p})\in\mathbb{N}^{p}$. (It means that there exist $s_{1},\dots,s_{p}\in\mathbb{N}$ such that the membership $f(r_{1},\dots, r_{p})\in\mathbb{Z}$ holds for all $(r_{1},\dots, r_{p})\in\mathbb{N}^{p}$ with $r_{1}\geq s_{1},\dots, r_{p}\geq s_{p}$.).
\end{definition}
It is clear that every polynomial with integer coefficients is numerical.  As an example of a numerical polynomial in $p$ variables with non-integer coefficients ($p\in\mathbb{N}, p\geq 1$) one can consider a polynomial
$\D\prod_{i=1}^{p}{t_{i}\choose m_{i}}$ \, where $m_{1},\dots, m_{p}\in\mathbb{N}$. (As usual, $\D{t\choose k}$ ($k\in\mathbb{Z}, k\geq 1$) denotes the polynomial $\D\frac{t(t-1)\dots (t-k+1)}{k!}$ in one
variable $t$, $\D{t\choose 0} = 1$, and $\D{t\choose k} = 0$ if $k$ is a negative integer.)
The following theorem proved in \cite[Chapter 2]{KLMP} gives the ''canonical'' representation of a numerical polynomial in several variables.

\begin{theorem}
Let $f(t_{1},\dots, t_{p})$ be a numerical polynomial in $p$ variables $t_{1},\dots,  t_{p}$, and let $\deg_{t_{i}}f = m_{i}$ ($1\leq i\leq p$) where $m_{1},\dots, m_{p}\in\mathbb{N}$. Then the
polynomial  $f(t_{1},\dots, t_{p})$ can be represented in the form

\begin{equation}
f(t_{1},\dots t_{p}) =\D\sum_{i_{1}=0}^{m_{1}}\dots \D\sum_{i_{p}=0}^{m_{p}}{a_{i_{1}\dots i_{p}}}{t_{1}+i_{1}\choose i_{1}}\dots{t_{p}+i_{p}\choose i_{p}}
\end{equation}
with integer coefficients $a_{i_{1}\dots i_{p}}$ ($0\leq i_{k}\leq m_{k}$ for $k = 1,\dots, p$) that are uniquely defined by the numerical polynomial.
\end{theorem}

In what follows (until the end of the section), we deal with subsets of the set $\mathbb{Z}^{m}$ ($m$ is a positive integer). Furthermore, we fix a partition of the set $\mathbb{N}_{m} = \{1,\dots , m\}$ into $p$ disjoint subsets ($p\geq 1$):
\begin{equation}
\mathbb{N}_{m} = \Delta_{1}\cup\Delta_{2}\cup\dots\Delta_{p}
\end{equation}
where $\Delta_{1} = \{1,\dots , m_{1}\}$, $\Delta_{2} = \{m_{1}+1,\dots , m_{1}+m_{2}\},\dots, \Delta_{p} = \{m_{1}+\dots +m_{p-1}+1,\dots , m\}$ ($m_{i} = \Card\Delta_{i}$ for $i=1,\dots, p$; $m_{1}+\dots +m_{p} = m$).

If $a = (a_{1},\dots, a_{m})\in {\bf Z}^{m}$, we denote the numbers $\D\sum_{i=1}^{m_{1}}|a_{i}|$, $\D\sum_{i=m_{1}+1}^{m_{1}+m_{2}}|a_{i}|,\dots,$ $\D\sum_{i=m_{1}+\dots + m_{p-1} + 1}^{m}|a_{i}|$ by $\ord_{1}a,\dots, \ord_{p}a$, respectively; $\ord_{k}a$ ($1\leq k\leq p$) is called the {\em order of $a$ with respect to $\Delta_{k}$}). Furthermore, we consider the set $\mathbb{Z}^{n}$ as the union

\begin{equation}
\mathbb{Z}^{m} = \bigcup_{1\leq j\leq 2^{m}}\mathbb{Z}_{j}^{(m)}
\end{equation}
where $\mathbb{Z}_{1}^{(m)}, \dots, \mathbb{Z}_{2^{m}}^{(m)}$ are all distinct Cartesian products of $m$ sets each of which is either $\mathbb{N}$ or $\mathbb{Z}_{\leq 0}$.
We assume that $\mathbb{Z}_{1}^{(m)} = \mathbb{N}$ and call $\mathbb{Z}_{j}^{(m)}$ the {\em $j$th orthant} of $\mathbb{Z}^{m}$ ($1\leq j\leq 2^{m}$).

The set $\mathbb{Z}^{m}$ will be considered as a partially ordered set with the order $\unlhd$ such that $(e_{1},\dots, e_{m})\unlhd (e'_{1},\dots, e'_{m})$ if and only if $(e_{1},\dots, e_{m})$ and $(e'_{1},\dots, e'_{m})$ lie in the same orthant $\mathbb{Z}_{k}^{m}$ and $(|e_{1}|,\dots, |e_{m}|)\leq_{P}(|e'_{1}|,\dots, |e'_{m}|)$.

In what follows, for any set $A\subseteq\mathbb{Z}^{m}$, $W_{A}$ will denote the set of all elements of $\mathbb{Z}^{m}$ that do not exceed any element of $A$ with respect to the order $\unlhd$. Furthermore, for any $r_{1},\dots r_{p}\in\mathbb{N}$, $A(r_{1},\dots r_{p})$ will denote the set of all elements $x = (x_{1},\dots, x_{m})\in A$ such that $\ord_{i}x\leq r_{i}$ ($i=1,\dots, p$).

The above notation can be naturally applied to subsets of $\mathbb{N}^{m}$ (treated as a subset of $\mathbb{Z}^{m}$). If $E\subseteq\mathbb{N}^{m}$ and $s_{1},\dots, s_{p}\in\mathbb{N}$, then $E(s_{1},\dots, s_{p})$ will denote the set of all $m$-tuples $e\in E$ such that $\ord_{i}e\leq s_{i}$ for $i=1,\dots, p$. Furthermore, we shall associate with a set $E\subseteq\mathbb{N}^{m}$ a set
$V_{E} =\{v\in\mathbb{N}^{m}\,|\,  v$ is not greater than or equal to any $m$-tuple in $E$ with respect to $\leq_{P}\}$. (Thus, $v=(v_{1}, \dots , v_{m})\in V_{E}$ if and only if for any element  $(e_{1},\dots , e_{m})\in E$, there exists $i\in\{1,\dots, m\}$ such that $e_{i} > v_{i}$.)

The following two theorems proved in \cite[Chapter 2]{KLMP} generalize the well-known Kolchin's result on the (univariate) numerical polynomials of subsets of $\mathbb{N}^{m}$ (see \cite[Chapter 0, Lemma 16]{Kolchin2}) and give explicit formulas for multivariate numerical polynomials associated with finite subsets of $\mathbb{N}^{m}$ and $\mathbb{Z}^{m}$.

\begin{theorem}
Let $E\subseteq\mathbb{N}^{m}$ and let partition (2) of $\mathbb{N}_{m}$ be fixed. Then there exists a numerical polynomial $\omega_{A}(t_{1},\dots, t_{p})$ such that

{\em (i)} \,  $\omega_{E}(r_{1},\dots, r_{p}) = \Card V_{A}(r_{1},\dots, r_{p})$ for all sufficiently large $(r_{1},\dots, r_{p})\in\mathbb{N}^{p}$.

{\em (ii)} \, The total degree $\deg\,\omega_{E}$ of the polynomial $\omega_{E}$ does not exceed $m$ and $\deg_{t_{i}}\omega_{E}\leq m_{i}$  ($1\leq i\leq p$).

{\em (iii)} \, $\deg\,\omega_{E} = m$ if and only if $E=\emptyset$. Then  $\omega_{E}(t_{1},\dots, t_{p}) = \D\prod_{i=1}^{p}{t_{i}+m_{i}\choose m_{i}}$.
\end{theorem}
\begin{definition}
The polynomial $\omega_{E}(t_{1},\dots, t_{p})$ is called the dimension polynomial of the set $E\subseteq\mathbb{N}^{m}$ associated with partition (2) of $\mathbb{N}_{m}$.
\end{definition}

\begin{theorem}
Let $E = \{e_{1}, \dots, e_{q}\}$ ($q\geq 1$) be a finite subset of $\mathbb{N}^{m}$ and let partition {\em (2)} of $\mathbb{N}_{m}$ be fixed. Let $e_{i} = (e_{i1}, \dots, e_{im})$ \, ($1\leq i\leq q$) and for any
$l\in\mathbb{N}$, $0\leq l\leq q$, let $\Theta(l,q)$ denote the set of all $l$-element subsets of $\mathbb{N}_{q} = \{1,\dots, q\}$. Let $\bar{e}_{\emptyset j} = 0$ and for any $\theta\in\Theta(l,q)$, $\theta\neq \emptyset$, let $\bar{e}_{\theta j} = \max \{e_{ij}\,|\,i\in \theta\}$  ($1\leq j\leq m$). Furthermore, let  $b_{\theta k} =\D\sum_{h\in\Delta_{k}}\bar{e}_{\theta h}$ ($k = 1,\dots, p$).
Then
\begin{equation}
\omega_{E}(t_{1},\dots, t_{p}) = \D\sum_{l=0}^{q}(-1)^{l}\D\sum_{\theta\in\Theta(l,\, q)}\D\prod_{j=1}^{p}{t_{j}+m_{j} - b_{\theta j}\choose m_{j}}.
\end{equation}
\end{theorem}
{\begin{remark} Clearly, if $E\subseteq\mathbb{N}^{m}$ and $E^{\ast}$ is the set of all minimal elements of $E$ with respect to the product order, then the set $E^{\ast}$ is finite and $\omega_{E}(t_{1}, \dots, t_{p}) = \omega_{E^{\ast}}(t_{1}, \dots, t_{p})$. Thus, the last theorem gives an algorithm that allows one to find the dimension polynomial of any subset of  $\mathbb{N}^{m}$  (and with a given partition (2) of $\mathbb{N}_{m}$): one should first find the set of all minimal points of the subset and then apply Theorem 2.6.
\end{remark}

The following theorem proved in \cite[Section 2.5]{KLMP} provides analogs of the results of Theorems 2.4 -  2.6 for subsets of $\mathbb{Z}^{m}$.

\begin{theorem}
Let $A\subseteq\mathbb{Z}^{m}$ and let partition {\em (2)} of the set $\mathbb{N}_{m}$ be fixed. Then there exists a numerical polynomial in $p$ variables $\phi_{A}(t_{1},\dots, t_{p})$  such that

{\em (i)}\, $\phi_{A}(r_{1},\dots, r_{p}) = \Card W_{A}(r_{1},\dots, r_{p})$ for all sufficiently large $p$-tuples $(r_{1},\dots, r_{p})\in\mathbb{N}^{p}$.

{\em (ii)}\, $\deg\phi_{A}\leq m$ and $\deg_{t_{i}}\phi_{A}\leq m_{i}$ for $i=1,\dots, p$. Furthermore, if the polynomial $\phi_{A}(t_{1},\dots, t_{p})$ is written in the form (1), then $2^{m}|a_{m_{1}\dots m_{p}}$.

{\em (iii)}\, Let us consider a mapping $\rho:\mathbb{Z}^{m}\longrightarrow\mathbb{N}^{2m}$ such that
$$\rho((e_{1},\dots, e_{m}) =(\max \{e_{1}, 0\}, \dots, \max \{e_{m}, 0\}, \max \{-e_{1}, 0\}, \dots, \max \{-e_{m}, 0\}).$$
Let $B = \rho(A)\bigcup \{\bar{e}_{1},\dots, \bar{e}_{m}\}$ where $\bar{e}_{i}$ ($1\leq i\leq m$) is a $2m$-tuple in $\mathbb{N}^{2m}$ whose $i$th and $(m+i)$th coordinates are equal to 1 and all other coordinates are equal to 0. Then
$$\phi_{A}(t_{1}, \dots, t_{p}) = \omega_{B}(t_{1}, \dots, t_{p})$$
where $\omega_{B}(t_{1},\dots,  t_{p})$ is the dimension polynomial of the set $B$ (see {\em Definition 2.5}) associated with the following partition of the set $\mathbb{N}_{2m}${\em :}
$\mathbb{N}_{2m} = \Delta'_{1}\cup\Delta_{2}\cup\dots\Delta'_{p}$ where $\Delta'_{i} = \Delta_{i}\cup\{m+k\,|\,k\in\Delta_{i}\}$ for $i=1,\dots, p$ (see partition {\em (2)}).

{\em (iv)}\, If $A = \emptyset$, then
\begin{equation}\label{eq:7}
\phi_{A}(t_{1}, \dots, t_{p}) =
\prod_{j=1}^{p}\left[\sum_{i=0}^{m_{j}}(-1)^{m_{j}-i}2^{i}{m_{j}\choose i}{{t_{j}+i}\choose i}\right].
\end{equation}
\end{theorem}
\begin{definition}
The polynomial $\phi_{A}(t_{1},\dots, t_{p})$ is called the {\bf dimension polynomial} of the set $A\subseteq\mathbb{Z}^{m}$ associated with partition (2) of $\mathbb{N}_{m}$.
\end{definition}

\begin{remark}
The equality (5) (as well as the last part of Theorem 2.1) expresses the fact that the number of solutions $(x_{1},\dots x_{m})\in\mathbb{Z}^{m}$ of the inequality $|x_{1}|+\dots + |x_{m}|\leq r$ ($r\in\mathbb{N}$) is $\D\sum_{k=0}^{m}(-1)^{m-k}2^{k}{n\choose k}{{r+k}\choose k}$ (see \cite[Proposition 2.1.9]{KLMP}). It follows that if $r_{1},\dots, r_{p}, s_{1},\dots, s_{p}\in\mathbb{N}$, $s_{i} < r_{i}$ ($1\leq i\leq p$), and  $B = \{b=(b_{1},\dots, b_{m})\in\mathbb{Z}^{m}\,|\,s_{i}\leq \sum_{\nu\in\Delta_{i}}|b_{\nu}|\leq r_{i}$ for $i=1,\dots, p\}$, (with the fixed partition (2) of $\mathbb{N}_{m}$), then
$$\Card B = \prod_{i=1}^{p}\left[\sum_{j=0}^{m_{i}}(-1)^{m_{i}-j}2^{j}{m_{i}\choose j}\left({{r_{i}+j}\choose j} - {{s_{i}+j-1}\choose j}\right)\right].$$
We will use this observation in the proof of Theorem 4.1.
\end{remark}

\section{$E$-reduction of inversive difference \\polynomials. $E$-characteristic sets}

Let $K$ be an inversive difference field with a basic set $\sigma = \{\alpha_{1},\dots, \alpha_{m}\}$.  Let us fix a partition of the set $\sigma$, that is, its representation as a union of $p$ disjoint subsets ($p\geq 1$):
\begin{equation}
\sigma = \sigma_{1}\bigcup \dots \bigcup \sigma_{p}\end{equation}
$$\text{where}\,\,\,  \sigma_{1} = \{\alpha_{1},\dots, \alpha_{m_{1}}\},\, \sigma_{2} = \{\alpha_{m_{1}+1},\dots, \alpha_{m_{1}+m_{2}}\}, \,\dots,$$
$$\sigma_{p} = \{\alpha_{m_{1}+\dots + m_{p-1}+1},\dots, \alpha_{m}\}\, \,\, \, (m_{1}+\dots + m_{p} = m).$$
If \,\,$\gamma  = \alpha_{1}^{k_{1}}\dots \alpha_{n}^{k_{n}}\in\Gamma$ ($k_{i}\in\mathbb{Z}$) then the order of $\gamma$ with respect to $\sigma_{i}$  ($1\leq i\leq p$) is defined as $\sum_{\nu = m_{1}+\dots + m_{i-1}+1}^{m_{1}+\dots + m_{i}}|k_{\nu}|$; it is denoted by $\ord_{i}\gamma$. If $i=1$, the last sum is replaced by $\sum_{\nu =1}^{m_{1}}|k_{\nu}|$. Furthermore, for any  $r_{1},\dots, r_{p}\in {\bf N}$, we set $\Gamma(r_{1},\dots, r_{p}) = \{\gamma\in\Gamma\,|\,\ord_{i}\gamma\leq r_{i} \,\,(i=1,\dots, p)\}.$

Let us consider $p$ total orderings $<_{1}, \dots, <_{p}$ of the group $\Gamma$ such that

\medskip

\noindent$\gamma  = \alpha_{1}^{k_{1}}\dots \alpha_{m}^{k_{m}}<_{i} \gamma'  = \alpha_{1}^{k'_{1}}\dots\alpha_{m}^{k'_{m}}$  ($1\leq i\leq p$)  if and only if the $(2m+p)$-tuple

\medskip

\noindent $(\ord_{i}\gamma, \ord_{1}\gamma,\dots, \ord_{i-1}\gamma, \ord_{i+1}\gamma, \dots, \ord_{p}\gamma, |k_{m_{1}+\dots + m_{i-1}+1}|,\dots, $\\
$|k_{m_{1}+\dots + m_{i}}|, k_{m_{1}+\dots + m_{i-1}+1},\dots,  k_{m_{1}+\dots,+ m_{i}}, |k_{1}|,\dots, |k_{m_{1}+\dots + m_{i-1}}|, $\\
$|k_{m_{1}+\dots + m_{i}+1}, \dots, |k_{m}|, k_{1},\dots, k_{m_{1}+\dots + m_{i-1}}, k_{m_{1}+\dots + m_{i}+1},\dots, k_{m})$

\noindent is less than the corresponding $(2m+p)$-tuple for $\gamma'$ with respect to the lexicographic order on $\mathbb{Z}^{2m+p}$.

Two elements $\gamma_{1} = \alpha_{1}^{k_{1}}\dots \alpha_{m}^{k_{m}}$ and $\gamma_{2} = \alpha_{1}^{l_{1}}\dots \alpha_{n}^{l_{m}}$ in $\Gamma$ are called {\em similar\/}, if the $m$-tuples $(k_{1}, \dots, k_{m})$ and $(l_{1}, \dots, l_{m})$ belong to the same orthant of $\mathbb{Z}^{m}$ (see (3)\,). In this case we write $\gamma_{1}\sim\gamma_{2}$. We say that $\gamma_{1}$ {\em divides} $\gamma_{2}$ (or $\gamma_{2}$ is a {\em multiple} of $\gamma_{1}$) and write $\gamma_{1}|\gamma_{2}$ if $\gamma_{1}\sim\gamma_{2}$ and there exists $\gamma\in\Gamma$ such that $\gamma\sim \gamma_{1}$ and $\gamma_{2} = \gamma\gamma_{1}$.

Let $R = K\{y_{1},\dots, y_{n}\}^{\ast}$ the algebra of $\sigma^{\ast}$-polynomials in $\sigma^{\ast}$-indeterminates $y_{1},\dots, y_{n}$ over $K$. Then $R$ can be viewed as a polynomial ring in the set of indeterminates $\Gamma Y = \{\gamma y_{i}\,|\,\gamma\in\Gamma, 1\leq i\leq n\}$ whose elements are called {\em terms}. For every $j=1,\dots, p$, we define the  order of a term $u=\gamma y_{i}$ with respect to $\sigma_{j}$ (denoted by $\ord_{j}u$)  as the corresponding order of $\gamma$. Furthermore, considering representation (3) of $\mathbb{Z}^{m}$ as the union of $2^{m}$ orthants $\mathbb{Z}_{j}^{m}$, we set
$\Gamma_{j} = \{\alpha_{1}^{k_{1}}\dots \alpha_{m}^{k_{m}}\in\Gamma\,|\,(k_{1},\dots, k_{m})\in\mathbb{Z}_{j}^{m}\}$ and $\Gamma_{j}Y = \{\gamma y_{i}\,|\,\gamma\in\Gamma_{j}, 1\leq i\leq n\}$.

Two terms $u=\gamma y_{i}$ and $v=\gamma' y_{j}$ are called {\em similar} if $\gamma$ and $\gamma'$ are similar; in this case we write $u\sim v$. If $u = \gamma y_{i}$ is a term and $\gamma'\in\Gamma$, we say that $u$ is similar to $\gamma'$ and write $u\sim\gamma'$ if $\gamma\sim\gamma'$. Clearly, if $u\in\Gamma Y$, $\gamma\in\Gamma$ and $\gamma\sim u$, then
$\ord_{j}(\gamma u) = \ord_{j}\gamma + \ord_{j}u$ for $j=1,\dots, p$. Furthermore, if $u, v\in\Gamma Y$, we say that $u$ {\em divides} $v$ (or $v$ is a {\em transform} or a {\em multiple} of $u$) and write $u\,|\,v$, if $u=\gamma' y_{i}$, $v=\gamma'' y_{i}$ for some $y_{i}$ and $\gamma'|\gamma''$. (If $\gamma'' = \gamma\gamma'$ for some $\gamma\in\Gamma, \,\gamma\sim\gamma'$, we write $\D\frac{v}{u}$ for $\gamma$.)

We consider $p$ orders $<_{1},\dots, <_{p}$ on the set $\Gamma Y$ that correspond to the orders on the group $\Gamma$ (we use the same symbols for the orders on $\Gamma$ and $\Gamma Y$). These orders are defined as follows: $\gamma y_{j} <_{i}\gamma' y_{k}$ if and only if $\gamma <_{i}\gamma'$ in $\Gamma$ or $\gamma = \gamma'$ and $j < k$ ($1\leq i\leq p,\, 1\leq j, k\leq n$).

\begin{definition}
Let $f\in K\{y_{1},\dots, y_{n}\}^{\ast}\setminus K$ and $1\leq k\leq p$. Then the greatest with respect to $<_{k}$ term that appears in $f$ is called the {\bf $k$-leader} of the $\sigma^{\ast}$-polynomial $f$; it is denoted by $u_{f}^{(k)}$.  The smallest with respect to $<_{k}$ term in $f$ is called the {\bf $k$-coleader} of $f$ and is denoted by $v_{f}^{(k)}$.
\end{definition}

\begin{definition} Let $f\in K\{y_{1},\dots, y_{n}\}\setminus K$ and let $u_{f}^{(k)} = \alpha_{1}^{k_{1}}\dots \alpha_{m}^{k_{m}}y_{i}$ and $v_{f}^{(k)} = \alpha_{1}^{l_{1}}\dots \alpha_{m}^{l_{m}}y_{j}$ be the $k$-leader and $k$-coleader of $f$, respectively ($1\leq k\leq p$). Then for every $k=1,\dots, p$, the nonnegative integer $\ord_{k}u_{f}^{(k)} - \ord_{k}v_{f}^{(k)}$ is called the {\bf $k$th effective order} of $f$; it is denoted by $\Eord_{k}f$.
\end{definition}

\begin{definition}
Let $f$ and $g$ be two $\sigma^{\ast}$-polynomials in the ring $K\{y_{1},\dots, y_{n}\}^{\ast}$. We say that {\em $f$ has lower rank than $g$} and write $\rk\,f < \rk\,g$ if either $f\in K$, $g\notin K$, or
$$(u_{f}^{(1)}, \deg_{u_{f}^{(1)}}f,  \ord_{2}u_{f}^{(2)}, \dots,\ord_{p}u_{f}^{(p)}, \Eord_{1}f,\dots, \Eord_{p}f) <_{\lex}\hspace{2in}$$
\begin{equation}
(u_{g}^{(1)}, \deg_{u_{g}^{(1)}}f,  \ord_{2}u_{g}^{(2)}, \dots,\ord_{p}u_{g}^{(p)}, \Eord_{1}g,\dots, \Eord_{p}g)\hspace{3in}
\end{equation}
(the comparison of $u_{f}^{(1)}$ and $u_{g}^{(1)}$ in this lexicographic order is made with respect to the order $<_{1}$ on the set of terms $\Gamma Y$). If the last two $(2p+1)$-tuples are equal (or $f, g\in K$) we say that $f$ and $g$ are of the same rank and write $\rk f = \rk g$.
\end{definition}
\begin{definition}
Let $f, g\in K\{y_{1},\dots, y_{n}\}^{\ast}$ and let $d = \deg_{u_{g}^{(1)}}g$. We say that $f$ is {\bf $E$-reduced} with respect to $g$ if one of the following two conditions holds.

(i)\, $f$ does not contain any $(\gamma u_{g}^{(1)})^{e}$  ($\gamma\in\Gamma$) such that $\gamma\sim u_{g}^{(1)}$ and $e\geq d$;

(ii)\, $f$ contains $(\gamma u_{g}^{(1)})^{e}$ with  some $\gamma\in\Gamma$, $\gamma\sim u_{g}^{(1)}$ and $e\geq d$, but in this case either there exists $k\in\mathbb{N}_{p}$, $k\geq 2$, such that
$\ord_{k}u_{\gamma g}^{(k)} > \ord_{k}(u_{f}^{(k)})$ or there exists $j\in \mathbb{N}_{p}$ such that $\ord_{j}v_{\gamma g}^{(j)} < \ord_{j}(v_{f}^{(j)})$. (The ``or'' here is inclusive, that is, the case when both conditions hold is included.)
\end{definition}
Thus, $f$ is not $E$-reduced with respect to $g$ if $f$ contains some $(\gamma u_{g}^{(1)})^{e}$  such that $\gamma\in\Gamma$, $\gamma\sim u_{g}^{(1)}$, $e\geq d = \deg_{u_{g}^{(1)}}g$,
$\ord_{k}u_{\gamma g}^{(k)}\leq\ord_{k}(u_{f}^{(k)})$ for $k=2,\dots, p$, and $\ord_{j}v_{\gamma g}^{(j)}\geq \ord_{j}(v_{f}^{(j)})$ for $j=1,\dots p$.

\begin{remark}
If $f, g\in K\{y_{1},\dots, y_{n}\}^{\ast}$ then $f$ is reduced with respect to $g$ in the sense of \cite[Definition 3.4.22]{KLMP} with respect to the term ordering $<_{1}$, if condition (i) of the last definition holds.  Clearly, in this case $f$ is $E$-reduced with respect to $g$.
\end{remark}

\begin{remark}
It follows from \cite[Lemma 3.3]{ZW} that for all $f\in R=K\{y_{1},\dots, y_{n}\}^{\ast}$, $j\in\{1,\dots, 2^{m}\}$ and $k\in\{1,\dots, p\}$,  there exist terms $u_{fjk}$ and $v_{fjk}$ in $f$ such that
for all elements $\gamma = \alpha_{1}^{k_{1}}\dots\alpha_{m}^{k_{m}}\in\Gamma_{j}$ with sufficiently large $(|k_{1}|,\dots, |k_{m}|)\in\mathbb{N}^{m}$ (in the sense of Definition 2.2), one has
$u_{\gamma f}^{(k)} = \gamma u_{fjk}$  and $v_{\gamma f}^{(k)} = \gamma v_{fjk}$. For example, let $\sigma = \{\alpha_{1}, \alpha_{2}, \alpha_{3}\}$ is considered with the partition
$\sigma = \sigma_{1}\cup\sigma_{2}\cup\sigma_{3}$ with $\sigma_{i} = \{\alpha_{i}\}$ ($i=1, 2, 3$), $f = \alpha_{1}^{2}\alpha_{2}^{-1}\alpha_{3}^{-3}y + \alpha_{1}^{-3}\alpha_{2}\alpha_{3}^{-4}y  +
\alpha_{1}\alpha_{2}^{-2}\alpha_{3}^{2}y + \alpha_{1}^{2}\alpha_{2}^{2}\alpha_{3}y\in K\{y\}^{\ast}$ and $\mathbb{Z}^{(3)}_{j} = \{(k_{1}, k_{2}, k_{3})\,|\,k_{1}\leq 0, k_{2}\geq 0, k_{3}\leq 0\}$. Then for any
$\gamma =  \alpha_{1}^{-r}\alpha_{2}^{s}\alpha_{3}^{-t}\in\Gamma_{j}$ ($r, s, t\geq 0$), we have $\gamma f = \alpha_{1}^{-r+2}\alpha_{2}^{s-1}\alpha_{3}^{-t-3}y + \alpha_{1}^{-r-3}\alpha_{2}^{s+1}\alpha_{3}^{-t-4}y  +
\alpha_{1}^{-r+1}\alpha_{2}^{s-2}\alpha_{3}^{-t+2}y + \alpha_{1}^{-r+2}\alpha_{2}^{s+}\alpha_{3}^{-t+1}y$, hence $u_{j1f} = u_{j3f} = \alpha_{1}^{-3}\alpha_{2}\alpha_{3}^{-4}y$, $u_{j2f} = v_{j1f} = \alpha_{1}^{2}\alpha_{2}^{2}\alpha_{3}y$, $v_{j2f} = v_{j3f} =\alpha_{1}\alpha_{2}^{-2}\alpha_{3}^{2}y$.

Therefore, if $f\in R$ and $u_{f}^{(1)} = \gamma_{1}y_{k}$ where $\gamma_{1}\in\Gamma_{j}$ ($1\leq j\leq 2^{m}$), then there exist $a_{if}, b_{kf}\in\mathbb{Z}$ ($2\leq i\leq p, 1\leq k\leq p$) such that for any
$\gamma\in\Gamma_{j}$, $\ord_{i}u_{\gamma f}^{(i)} = \ord\gamma + a_{if}$ and $\ord_{k}v_{\gamma f}^{(k)} = \ord\gamma + b_{kf}$.
\end{remark}

\begin{proposition}
If $f, g\in K\{y_{1},\dots, y_{n}\}^{\ast}$ and $\rk f < \rk g$, then $f$ is $E$-reduced with respect to $g$.
\end{proposition}

\begin{proof} Suppose that $f$ is not $E$-reduced with respect to $g$. If $f$ contains some $(\gamma u_{g}^{(1)})^{e}$  such that $\gamma\in\Gamma$, $\gamma\sim u_{g}^{(1)}$, and $e\geq d = \deg_{u_{g}^{(1)}}g$, then $\gamma = 1$ (if $\gamma\neq 1$, then $u_{g}^{(1)} <_{1} \gamma u_{g}^{(1)}\leq_{1}u_{f}^{(1)}$ that contradicts the condition (7) for $\rk f < \rk g$). Now the fact that $f$ is not $E$-reduced with respect to $g$ implies that
$\ord_{k}u_{g}^{(k)}\leq\ord_{k}u_{f}^{(k)}$ for $k=2,\dots, p$ and $\ord_{k}v_{g}^{(k)}\geq\ord_{k}v_{f}^{(k)}$ for $k=1,\dots, p$. It follows that $\Eord_{k}g\leq \Eord_{k}f$ ($1\leq k\leq p$), so we have arrived at a contradiction with the inequality $\rk f < \rk g$. Therefore, $f$ is $E$-reduced with respect to $g$.
\end{proof}

\begin{proposition}
Let $\mathcal{A} =\{g_{1},\dots, g_{t}\}$ be a finite set of $\sigma^{\ast}$-polynomials in the ring $R=K\{y_{1},\dots, y_{n}\}^{\ast}$, let $u^{(i)}_{k}$ and $v^{(i)}_{k}$ denote the $i$-leader and $i$-coleader of $g_{k}$, respectively ($1\leq k\leq t, 1\leq i\leq p$). Let $d_{k} = \deg_{u^{(1)}_{k}}g_{k}$ and $I_{k}$ denote the coefficient of $(u^{(1)}_{k})^{d_{k}}$ when $g_{k}$ is written as a polynomial in $u^{(1)}_{k}$ ($1\leq k\leq t$). Furthermore, let
$I({\mathcal{A}}) = \{f\in R\,|\,$ either $f =1$ or $f$ is a product of finitely many $\sigma^{\ast}$-polynomials of the form  $\gamma(I_{k})$ ($\gamma\in\Gamma, k=1,\dots, t)\}$. Then for any $h\in R$,
there exist $J\in I(\mathcal{A})$ and  $\overline{h}\in R$ such that $\overline{h}$ is $E$-reduced with respect to $\mathcal{A}$ and $Jh\equiv \overline{h} (mod\, [\mathcal{A}]^{\ast})$ (that is, $Jh-\overline{h}\in [\mathcal{A}]^{\ast}$).
\end{proposition}

\begin{proof}
If $h$ is $E$-reduced with respect to $\mathcal{A}$, the statement is obvious (one can set $\overline{h}=h$). Suppose that $h$ is not $E$-reduced with respect to $\mathcal{A}$. In what follows, if a $\sigma$-polynomial $f\in R$ is not $E$-reduced with respect to $\mathcal{A}$, then  a term $w_f$ that appears in $f$ will be called the {\em $\mathcal{A}$-leader} of $f$ if $w_f$ is the greatest (with respect to $<_{1}$) term among all terms of the form
$\gamma u^{(1)}_{g_k}$ with $\gamma\in\Gamma, \gamma\sim u^{(1)}_{g_k}$, ($1\leq k\leq t$) such that $f$ contains $(\gamma u^{(1)}_{k})^{e}$ with $e\geq d_{k}$, $\ord_{i}u_{\gamma g_{k}}^{(i)}\leq \ord_{i}u^{(i)}_{f}$ for $i=2,\dots, p$, and $\ord_{j}v_{\gamma g_{k}}^{(i)}\geq \ord_{j}v^{(j)}_{f}$ for $j=1,\dots, p$.

Let $w_h$ be the $\mathcal{A}$-leader of the element $h$, $d=\deg_{w_h}h$, and $c_h$ the coefficient of $w^{d}_h$ when $h$ is written as a polynomial in $w_h$. Then $w_{h}=\gamma u^{(1)}_{k}$ for some $k\in\{1,\dots, t\}$ and   $\gamma\in\Gamma$ such that $\gamma\sim u^{(1)}_{g_k}$, $d\geq d_{k}$, $\ord_{i}u_{\gamma g_{k}}^{(i)}\leq \ord_{i}u^{(i)}_{h}$ ($2\leq i\leq p$), and $\ord_{j}v_{\gamma g_{k}}^{(j)}\geq \ord_{j}v^{(j)}_{h}$ ($1\leq j\leq p$).

Let us choose such $k$ that corresponds to the maximum (with respect to $<_{1}$) $1$-leader $u^{(1)}_{i}$ ($1\leq i\leq t$) and consider the $\sigma^{\ast}$-polynomial $h' = \gamma(I_{k})h - c_{h}w_{h}^{d-d_{k}}(\gamma g_{k})$. Clearly, $\deg_{w_{h}}h' < \deg_{w_{h}}h$ and $h'$ does not contain any $\mathcal{A}$-leader $\gamma' u^{(1)}_{\nu}$ ($\gamma'\in\Gamma, 1\leq\nu\leq t$) that is greater than $w_{h}$ with respect to $<_{1}$  (such a term cannot appear in $\gamma(I_{k})h$ or $\gamma g_{k}$, since $u_{\gamma g_{k}}^{(1)} = \gamma u^{(1)}_{g_k} = w_{h}$). Applying the same procedure to $h'$ and continuing in the same way, we will arrive at a $\sigma$-polynomial $\overline{h}\in R$ such that $\overline{h}$ is $E$-reduced with respect to $\mathcal{A}$ and $Jh-\overline{h}\in [\mathcal{A}]^{\ast}$ for some $J\in I({\mathcal{A}})$.
\end{proof}

The process of reduction described in the proof of the last proposition can be realized by the following algorithm. (Recall that $\mathcal{E}_{R}$ denotes the ring of $\sigma^{\ast}$-operators over the $\sigma^{\ast}$-ring
$R = K\{y_{1},\dots, y_{n}\}^{\ast}$.)

\begin{alg}  ($h, t, g_{1},\dots, g_{t};\, \overline{h}$)

{\bf Input:} $h\in R$, a positive integer $t$, $\mathcal{A} = \{g_{1},\dots, g_{t}\}\subseteq R$ where $g_{i}\ne 0$  for $i = 1,\dots, t$

{\bf Output:}  Element $\overline{h}\in R$, elements $C_{1},\dots, C_{t}\in\mathcal{E}_{R}$ and $J\in I({\mathcal{A}})$ such that
$Jh = \sum_{i=1}^{t}C_{i}(g_{i}) + \overline{h}$  and $\overline{h}$ is $E$-reduced with respect to $\mathcal{A}$

{\bf Begin}

$C_{1}:= 0, \dots, C_{t}:= 0, \overline{h}:= h$

{\bf While} there exist $k$, $1\leq k\leq t$, and a term $w$ that appears in $\overline{h}$  with a (nonzero) coefficient $c_{w}$, such that
$u^{(1)}_{g_k}\,|\,w$,  $\deg_{u^{(1)}_{g_k}}g_{k}\leq\deg_{w}\overline{h}$, $\ord_{i}(\gamma_{kw}u^{(i)}_{g_k})\leq\ord_{i}u^{(i)}_{\overline{h}}$ for $i=2,\dots, p$, where $\gamma_{kw} = {\frac{w}{u^{(1)}_{g_{k}}}}$, and                $\ord_{j}(\gamma_{kw}v^{(j)}_{g_k})\geq \ord_{j}v^{(j)}_{\overline{h}}$  for $j=1, \dots, p$, {\bf do}

$z$:= the greatest of the terms $w$ that satisfy the above conditions.

$l$:= the smallest number $k$  for which $u^{(1)}_{g_k}$ is the greatest (with resect to $<_{1}$) $1$-leader of an element of $\mathcal{A}$  such that
$u^{(1)}_{g_k}\,|\,z$,  $\deg_{u^{(1)}_{g_k}}g_{k}\leq\deg_{z}\overline{h}$,  $\ord_{i}(\gamma_{kz}u^{(i)}_{g_k})\leq\ord_{i}u^{(i)}_{\overline{h}}$ for $i=2,\dots, p$, where $\gamma_{kz} = {\frac{z}{u^{(1)}_{g_{k}}}}$, and $\ord_{j}(\gamma_{kz}v^{(j)}_{g_k})\geq\ord_{j}v^{(j)}_{\overline{h}}$  for $j=1, \dots, p$,

$J:= \gamma(I_{l})J, C_{l}:= C_{l} + c_{z}z^{d-d_{l}}\gamma_{lz}$ where $d=\deg_{z}\overline{h}$, $d_{l}=\deg_{u^{(1)}_{g_l}}g_{l}$, and $c_{z}$ is the coefficient of $z^{d}$ when $\overline{h}$ is written as a polynomial in $z$

$\overline{h}:= \tau(I_{l})h^{\ast} - c_{z}z^{d-d_{l}}(\gamma g_{l})$

{\bf End}

\end{alg}

\begin{definition}
A set $\mathcal{A}\subseteq K\{y_{1},\dots, y_{n}\}^{\ast}$ is said to be {\bf $E$-autoreduced} if either it is empty or $\mathcal{A}\bigcap K = \emptyset$ and every element of $\mathcal{A}$ is $E$-reduced with
respect to all other elements of the set $\mathcal{A}$.
\end{definition}

\begin{example}
Let $K$ be an inversive difference field with a basic set $\sigma = \{\alpha_{1}, \alpha_{2}\}$ considered with a partition $\sigma = \sigma_{1}\cup\sigma_{2}$ where $\sigma_{1} = \{\alpha_{1}\}$ and  $\sigma_{2} = \{\alpha_{2}\}$. Let $\mathcal{A} = \{g, h\}\subseteq K\{y\}^{\ast}$ (the ring of $\sigma^{\ast}$-polynomials in one $\sigma^{\ast}$-indeterminate $y$) where
$$g = \alpha_{1}^{3}\alpha_{2}^{-2}y + \alpha_{2}^{3}y + \alpha_{2}y,\hspace{0.3in} h = \alpha_{1}^{2}\alpha_{2}^{-1}y + \alpha_{1}^{-1}\alpha_{2}^{2}y + \alpha_{1}\alpha_{2}y.$$
Then $u_{g}^{(1)} = \alpha_{1}^{3}\alpha_{2}^{-2}y$, $v_{g}^{(1)} = u_{g}^{(2)} = \alpha_{2}^{3}y$, $v_{g}^{(2)} = \alpha_{2}y$, $u_{h}^{(1)} = \alpha_{1}^{2}\alpha_{2}^{-1}y$, $v_{h}^{(1)} = v_{h}^{(2)} = \alpha_{1}\alpha_{2}y$, and $u_{h}^{(2)} = \alpha_{1}^{-1}\alpha_{2}^{2}y$. We see that $u_{g}^{(1)}$ is a transform of $u_{h}^{(1)}$, $u_{g}^{(1)} = \gamma u_{h}^{(1)}$ where $\gamma =  \alpha_{1}\alpha_{2}^{-1}\sim u_{h}^{(1)}$. Furthermore, $\gamma h = \alpha_{1}^{3}\alpha_{2}^{-2}y + \alpha_{1}^{2}y + \alpha_{2}y$, so $u_{\gamma h}^{(1)} = u_{\gamma h}^{(2)} = \alpha_{1}^{3}\alpha_{2}^{-2}y$, $v_{\gamma h}^{(1)} = \alpha_{2}y$,
and $v_{\gamma h}^{(2)} = \alpha_{1}^{2}y$. Thus, $\ord_{2}u_{\gamma h}^{(2)} = 2 < \ord_{2}u_{g}^{(2)} = 3$, $\ord_{1}v_{\gamma h}^{(1)} = 0  = \ord_{1}v_{g}^{(1)}$, but
$\ord_{2}v_{\gamma h}^{(2)} = 0 < \ord_{2}v_{g}^{(2)} = 1$. Therefore, $g$ is $E$-reduced with respect to $h$. Since $h$ is clearly $E$-reduced with respect to $g$, $\mathcal{A} = \{g, h\}$ is an $E$-autoreduced set.
At the same time, this set is not autoreduced in the sense of \cite{L6} where an analog of Definition 3.4 does not require the option ``there exists $j\in \mathbb{N}_{p}$ such that $\ord_{j}v_{\gamma g}^{(j)} < \ord_{j}(v_{f}^{(j)})$'' in the case when $f$ contains $(\gamma u_{g}^{(1)})^{e}$ with some $\gamma\in\Gamma$, $\gamma\sim u_{g}^{(1)}$ and $e\geq d$ (see Definition 3.4).
\end{example}

We are going to show that every $E$-autoreduced set is finite. The proof of the following lemma can be found in \cite[Chapter 0, Section 17]{Kolchin2}.

\begin{lemma}
Let $A$ be an infinite subset of the set $\mathbb{N}^{m}\times\mathbb{N}_{n}$ ($m, n\geq 1$). Then there exists an infinite sequence of elements of $A$, strictly increasing relative to the product order,
in which every element has the same projection on $\mathbb{N}_{n}$.
\end{lemma}

Since every infinite sequence of elements of $\Gamma$ contains an infinite subsequence whose elements are similar to each other (there are only finitely many orthants of $\mathbb{Z}^{m}$), the last lemma immediately implies the following statement that will be used below.

\begin{lemma}
Let $S$ be any infinite set of terms $\gamma y_j$ ($\gamma\in\Gamma, 1\leq j\leq n$) in the ring $K\{y_{1},\dots, y_{n}\}^{\ast}$.   Then there exists an index $j$ ($1\leq j\leq n$) and an infinite sequence of terms
$\gamma_{1}y_{j}, \gamma_{2}y_{j}, \dots, \gamma_{k}y_{j},\dots $ in $S$ such that $\gamma_{k}\,|\,\gamma_{k+1}$ for every $k=1, 2, \dots $.
\end{lemma}

\begin{proposition}
Every $E$-autoreduced set is finite.
\end{proposition}
\begin{proof}
Suppose that there is an infinite $E$-autoreduced set $\mathcal{A}$. It follows from Lemma 3.11 that $\mathcal{A}$ contains a sequence of $\sigma^{\ast}$-polynomials $\{f_{1}, f_{2},\dots\}$ such that  $u^{(1)}_{f_{i}}\,|\,u^{(1)}_{f_{i+1}}$  for $i = 1, 2,\dots$. Since the sequence of non-negative integers $\{\deg_{u^{(1)}_{f_{i}}}f_{i}\}$ cannot have an infinite decreasing subsequence, without loss of generality we can assume that $\deg_{u^{(1)}_{f_{i}}}f_{i}\leq \deg_{u^{(1)}_{f_{i+1}}}f_{i+1}$ ($i=1, 2,\dots$).

Let $k_{ij} = \ord_{j}u_{f_{i}}^{(1)}$, $l_{ij} = \ord_{j}u_{f_{i}}^{(j)}$, $n_{ij} = \ord_{j}v_{f_{i}}^{(j)}$ ($1\leq j\leq p$). Obviously, \,$l_{ij}\geq k_{ij}\geq n_{ij}$\, ($i = 1, 2,\dots ; j = 1,\dots, p$),\, so
$\{(l_{i1}-k_{i1}=0, l_{i2}-k_{i2}, \dots, l_{ip}-k_{ip})\,|\,i =1, 2, \dots \}\subseteq\mathbb{N}^{p}$ and $\{(k_{i1}-n_{i1}, k_{i2}-n_{i2}, \dots, k_{ip}-n_{ip}) \mid i =1, 2, \dots \}\subseteq\mathbb{N}^{p}$. By Lemma 3.10,  there exists an infinite sequence of indices $i_{1} < i_{2} < \dots $ such that
\begin{equation}
(l_{i_{1}2}-k_{i_{1}2},\dots, l_{i_{1}p}-k_{i_{1}p}) \leq_{P} (l_{i_{2}2}-k_{i_{2}2},\dots, l_{i_{2}p}-k_{i_{2}p}) \leq_{P} \dots
\end{equation}
and
\begin{equation}
(k_{i_{1}1}-n_{i_{1}1},\dots, k_{i_{1}p}-n_{i_{1}p}) \leq_{P} (k_{i_{1}1}-n_{i_{1}1},\dots, k_{i_{2}p}-n_{i_{2}p}) \leq_{P} \dots .
\end{equation}

Then for any $j = 2,\dots, p$ and for $\gamma_{12} = \D\frac{u_{f_{i_{2}}}^{(1)}}{u_{f_{i_{1}}}^{(1)}}$, we have (using (8)) $\ord_{j}u_{\gamma_{12}f_{i_{1}}}^{(j)}\leq \ord_{j}\gamma_{12}u_{f_{i_{1}}}^{(j)} =                              k_{i_{2}j} - k_{i_{1}j} + l_{i_{1}j} \leq k_{i_{2}j} + l_{i_{2}j} - k_{i_{2}j} = l_{i_{2}j} = \ord_{j}u_{f_{i_{2}}}^{(j)}$. Similar arguments with the use of (9) show that
$\ord_{j}(\tau v_{f_{i_{1}}}^{(j)}) \geq \ord_{j}v_{f_{i_{2}}}^{(j)}$ for $j = 2, \dots, p$. Thus, the $\sigma^{\ast}$-polynomial $f_{i_{2}}$ is not $E$-reduced with respect to $f_{i_{1}}$ that contradicts the fact that
$\mathcal{A}$ is an $E$-autoreduced set.
\end{proof}
In what follows, while considering $E$-autoreduced sets we always assume that their elements are arranged in order of increasing rank.

\begin{definition}
Let  $\mathcal{A} = \{g_{1},\dots, g_{s}\}$ and $\mathcal{B} = \{h_{1},\dots, h_{t}\}$ be two $E$-autoreduced sets in the ring $K\{y_{1},\dots, y_{n}\}^{\ast}$. Then  $\mathcal{A}$ is said to have lower rank than $\mathcal{B}$, written as
$\rk \mathcal{A} <  \rk \mathcal{B}$, if one of the following two cases holds:

(1)\, $\rk g_{1} < \rk h_{1}$ or there exists $k\in\mathbb{N}$ such that $1 < k\leq \min \{s, t\}$, $\rk g_{i}=\rk h_{i}$ for $i=1,\dots,k-1$ and  $\rk g_{k} < \rk h_{k}$.

(2)\, $s > t$ and  $\rk g_{i} = \rk h_{i}$ for $i=1,\dots, t$.

If $s=t$ and $\rk g_{i} = \rk h_{i}$ for $i=1,\dots, s$, then $\mathcal{A}$ is said to have the same rank as $\mathcal{B}$; in this case we write $\rk \mathcal{A} = \rk \mathcal{B}$
\end{definition}

\begin{proposition}
In every nonempty family of $E$-autoreduced sets of difference polynomials there exists an $E$-autoreduced set of lowest rank.
\end{proposition}

\begin{proof}
In order to proof the proposition, we will mimic the proof of the corresponding statement for differential polynomials, see \cite[Chapter 1, Proposition 3]{Kolchin2}  as follows. Let $\mathcal{M}$ be a nonempty family of $E$-autoreduced sets in the ring  $K\{y_{1},\dots, y_{n}\}^{\ast}$. Let us inductively define an infinite descending chain of subsets of $\mathcal{M}$ as follows:
$\mathcal{M}_{0}=\mathcal{M}$, $\mathcal{M}_{1}=\{\mathcal{A}\in \mathcal{M}_{0}\,|\,\mathcal{A}$ contains at least one element and the first element of $\mathcal{A}$ is of lowest possible rank$\}, \dots ,
\mathcal{M}_{k}=\{\mathcal{A}\in \mathcal{M}_{k-1}\,|\,\mathcal{A}$ contains at least $k$ elements and the $k$th element of $\mathcal{A}$ is of lowest possible rank$\}, \dots $. It is clear that if $\mathcal{A}$ and $\mathcal{B}$ are any two $E$-autoreduced sets in $\mathcal{M}_{k}$ and $f$ and $g$ are their $l$th $\sigma$-polynomials ($l\geq k$), then $\rk f = \rk g$. Therefore, if all sets $\mathcal{M}_{k}$ are nonempty, then the set $\{A_{k}\,|\,A_{k}$ is the $k$th element of some $E$-autoreduced set in $\mathcal{M}_{k}\}$ would be an infinite $E$-autoreduced set, and this would contradict Proposition 3.12. Thus, there is the smallest
positive integer $k$ such that $\mathcal{M}_{k}=\emptyset$. Clearly, every element of $\mathcal{M}_{k-1}$ is an $E$-autoreduced set of lowest rank in the family $\mathcal{M}$.
\end{proof}

Let $J$ be any nonzero ideal of the ring   $K\{y_{1},\dots, y_{n}\}^{\ast}$. Since the set of all $E$-autoreduced subsets of $J$ is not empty (if $0\neq f\in J$, then $\{f\}$ is an  $E$-autoreduced subset of $J$), the last statement shows that  $J$ contains an  $E$-autoreduced subset of lowest rank. Such an  $E$-autoreduced set is called an {\bf  $E$-characteristic set} of the ideal $J$.

\begin{proposition}
Let $\mathcal{A} = \{f_{1}, \dots , f_{d}\}$ be an  $E$-characteristic set of a $\sigma$-ideal $J$ of the ring  $K\{y_{1},\dots, y_{n}\}^{\ast}$.  Then an element $g\in J$ is $E$-reduced with respect to the set $\mathcal{A}$ if and only if $g = 0$.
\end{proposition}

\begin{proof}  First of all, note that if $g\neq 0$ and $\rk\,g < \rk\,f_{1}$, then $\rk\,\{g\} < \rk\,\mathcal{A}$ that contradicts the fact that $\mathcal{A}$ is a $E$-characteristic set of the ideal $J$. Let $\rk\,g > \rk\,f_{1}$ and let $f_{1},\dots, f_{j}$ ($1\leq j\leq d$) be all elements of $\mathcal{A}$ whose rank is lower that the rank of $g$. Then the set $\mathcal{A}' = \{f_{1},\dots, f_{j}, g\}$ is $E$-autoreduced. Indeed,
by the conditions of the proposition, $\sigma$-polynomials $f_{1},\dots, f_{j}$ are $E$-reduced with respect to each other and $g$ is $E$-reduced with respect to the set $\{f_{1},\dots, f_{j}\}$.  Furthermore, each
$f_{i}$  ($1\leq i\leq j$) is $E$-reduced with respect to $g$ because $\rk\,f_{i} < \rk\,g$. Since $\rk\,\mathcal{A}' < \rk\,\mathcal{A}$, $\mathcal{A}$ is not an $E$-characteristic set of $J$ that contradicts the conditions of the proposition. Thus, $g = 0$.
\end{proof}

It follows from Remark 3.5 that every autoreduced (respectively, characteristic) set of an ideal $J$ of $K\{y_{1},\dots, y_{n}\}^{\ast}$ in the sense of \cite[Definitions 3.4.23 and 3.4.31]{KLMP} with respect to $<_{1}$ is an $E$-autoreduced (respectively, $E$-characteristic) set of $J$. Therefore, one can apply \cite[Corollary 6.5.4]{KLMP} to obtain the following statement.

\begin{proposition}
Let $\preceq$ be a preorder on $K\{y_{1},\dots, y_{n}\}^{\ast}$ such that $f\preceq g$ if and only if $u_{g}^{(1)}$ is a transform of $u_{f}^{(1)}$. Let $f$ be a linear $\sigma^{\ast}$-polynomial in $K\{y_{1},\dots, y_{n}\}^{\ast}\setminus K$. Then the set of all minimal with respect to $\preceq$  elements of the set $\{\gamma f\,|\,\gamma\in\Gamma\}$ is an $E$-characteristic set of the $\sigma^{\ast}$-ideal $[f]^{\ast}$.
\end{proposition}

\section{A new type of multivariate dimension polynomials of $\sigma^{\ast}$-field extensions}

In this section we use properties of $E$-characteristic sets to obtain the following result that generalizes Theorem 2.1 and introduces a new type of multivariate dimension polynomials of finitely generated inversive difference field extensions that carry more invariants than any previously known difference dimension polynomials. (By an invariant of an inversive difference ($\sigma^{\ast}$-) field extension we mean a numerical characteristic that does not depend on the choice of the finite set of its $\sigma^{\ast}$-generators.)  As before, $K$ denotes an inversive difference ($\sigma^{\ast}$-) field with a basic set $\sigma = \{\alpha_{1},\dots, \alpha_{m}\}$ considered together with its partition (6) into the union of $p$ disjoint subsets $\sigma_{i}$, $\Card\sigma_{i} = m_{i}$ ($1\leq i\leq p$). Furthermore, for any two $p$-tuples
$(r_{1},\dots, r_{p}), (s_{1},\dots, s_{p})\in\mathbb{N}^{p}$ with $s_{i}\leq r_{i}$ for $i=1,\dots, p$, we set
$$\Gamma(r_{1}, \dots, r_{p}; s_{1},\dots, s_{p}) = \{\gamma\in\Gamma\,|\,s_{i}\leq\ord_{i}\gamma\leq r_{i}\,\,\text{for}\,\, i=1,\dots, p\}.$$

\begin{theorem}
Let $L = K\langle \eta_{1},\dots,\eta_{n}\rangle^{\ast}$ be a $\sigma^{\ast}$-field extension generated by a set $\eta = \{\eta_{1}, \dots , \eta_{n}\}$. Then there exists a polynomial
$\Phi_{\eta\,|\,K}(t_{1},\dots, t_{2p})$ in $2p$ variables with rational coefficients and numbers $r^{(0)}_{i}, s_{i}^{(0)}, s^{(1)}_{i}\in\mathbb{N}$ ($1\leq i\leq p$) with $s_{i}^{(1)} < r^{(0)}_{i}-s_{i}^{(0)}$ such that
$$\Phi_{\eta\, |\,K}(r_{1},\dots, r_{p}, s_{1},\dots, s_{p}) = \hspace{4in}$$ $$\trdeg_{K}K(\{\gamma\eta_{j}\,|\,\gamma\in\Gamma(r_{1}, \dots, r_{p}; s_{1},\dots, s_{p}), 1\leq j\leq n\})\hspace{4in}$$
for all $(r_{1},\dots, r_{p}, s_{1},\dots, s_{p})\in\mathbb{N}^{2p}$ with $r_{i}\geq r_{i}^{(0)}$, $s_{i}^{(1)}\leq s_{i}\leq r_{i}-s_{i}^{(0)}$. Furthermore,
$\deg\Phi_{\eta\,|\,K}\leq m$, $\deg_{t_{i}}\Phi_{\eta\,|\,K}\leq m_{i}$ for $i=1,\dots, p$ and $\deg_{t_{j}}\Phi_{\eta\,|\,K}\leq m_{j-p}$ for $j=p+1,\dots, 2p$.
\end{theorem}

\begin{proof}
Let $P\subseteq R=K\{y_{1},\dots, y_{n}\}$ be the defining $\sigma^{\ast}$-ideal of the extension $L/K$ and let ${\mathcal{A}} = \{f_{1},\dots, f_{q}\}$ be an $E$-characteristic set of $P$.  For any $\overline{r} = (r_{1},\dots, r_{p}), \overline{s} = (s_{1},\dots, s_{p})\in\mathbb{N}^{p}$ such that $\overline{s}\leq_{P}\overline{r}$ (that is, $s_{i}\leq r_{i}$ for $i=1,\dots, p$), let
$$W(\overline{r}, \overline{s}) = \{w\in\Gamma Y\,|\,s_{i}\leq \ord_{i}w\leq r_{i}\,\,\, \text{for}\,\,\, i=1,\dots, p \},\,\,\, \hspace{3in}$$
$$W_{\eta}(\overline{r}, \overline{s}) = \{w(\eta)\,|\,w\in W(\overline{r}, \overline{s})\},\hspace{6in}$$
$$U'(\overline{r}, \overline{s}) = \{u\in\Gamma Y\,|\,s_{i}\leq \ord_{i}u\leq r_{i}\,\,\,\, \text{for $i=1,\dots, p$ and $u$ is not a transform}\hspace{1in}$$
$$\text{of any}\,\,\, u^{(1)}_{f_{j}} \, (1\leq j\leq q)\},\hspace{5in}$$
$$U'_{\eta}(\overline{r}, \overline{s}) = \{u(\eta)\,|\,u\in U'(\overline{r}, \overline{s})\},\hspace{6in}$$
$$U''(\overline{r}, \overline{s}) = \{u\in\Gamma Y\,|\,s_{i}\leq \ord_{i}u\leq r_{i} \,(1\leq i\leq p),\,\,\, \text{$u$ is a transform of some $u^{(1)}_{f_{j}}$}\hspace{1in} $$
($1\leq j\leq q$) and whenever $u=\gamma u^{(1)}_{f_{j}}$ ($\gamma\in\Gamma$, $\gamma\sim u^{(1)}_{f_{j}})$, either $\ord_{1}v^{(1)}_{\gamma f_{j}} < s_{1}$ or there exists $k\in\{2,\dots, p\}$ such that
$\ord_{k}(u^{(k)}_{\gamma f_{j}})> r_{k}$ or there exists $i\in\{2,\dots, p\}$ such that $\ord_{i}v^{(i)}_{\gamma f_{j}} < s_{i}$ (``or'' is inclusive)$\}$,
$$U_{\eta}''(\overline{r}, \overline{s}) = \{u(\eta)\,|\,u\in U''(\overline{r}, \overline{s})\}.\hspace{5in}$$
Furthermore, let
$$U(\overline{r}, \overline{s}) = U'(\overline{r}, \overline{s})\cup U''(\overline{r}, \overline{s})\,\,\, \text {and}\,\,\, U_{\eta}(\overline{r}, \overline{s}) = U'_{\eta}(\overline{r}, \overline{s})\cup
U''_{\eta}(\overline{r}, \overline{s}).$$
We are going to prove that for every $\overline{r}, \overline{s}\in\mathbb{N}^{p}$ with $\overline{s} <_{P}\overline{r}$, the set $U_{\eta}(\overline{r}, \overline{s})$ is a transcendence basis of the field
$K(W_{\eta}(\overline{r}, \overline{s}))$ over $K$.

First, one can see that this set is algebraically independent over $K$. Indeed, if $f(w_{1}(\eta),\dots, w_{k}(\eta)) = 0$ for some elements $w_{1},\dots, w_{k}\in U(\overline{r}, \overline{s})$, then the $\sigma^{\ast}$-polynomial $f(w_{1},\dots, w_{k})$ lies in $P$ and it is $E$-reduced with respect to ${\mathcal{A}}$. (If $f$ contains a term $w = \gamma u^{(1)}_{f_{j}}$, $1\leq i\leq q$, $\gamma\in\Gamma$,
$\gamma\sim u^{(1)}_{f_{j}}$ such that $\deg_{w}f\geq\deg_{u^{(1)}_{f_{j}}}f_{j}$, then $w\in U''(\overline{r}, \overline{s})$, so either $\ord_{1}(v^{(1)}_{\gamma f_{j}}) < s_{1}\leq\ord_{1}v^{(1)}_f$ or
there exist $k\in\{2,\dots, q\}$ such that $\ord_{k}u_{\gamma f_{j}}^{(k)} > r_{k}\geq \ord_{k}u_{f}^{(k)}$ or there exists $i\in\{2,\dots, p\}$ such that $\ord_{i}v_{\gamma f_{j}}^{(i)} < s_{i}\leq \ord_{i}v_{f}^{(i)}$;
``or'' is inclusive). It follows that $f$ is $E$-reduced with respect to ${\mathcal{A}}$.) By Proposition 3.15, $f=0$, so the set $U_{\eta}(\overline{r}, \overline{s})$ is algebraically independent over $K$.

Now let us prove that if $0\leq s_{i}\leq r_{i}-s^{(0)}_{i}$, where $s^{(0)}_{i}=\max\{\Eord_{i}f_{j}\,|\,1\leq j\leq q\}$ ($1\leq i\leq p$), then every element
$\gamma\eta_{k}\in W_{\eta}(\overline{r}, \overline{s})\setminus U_{\eta}(\overline{r}, \overline{s})$ ($\gamma\in\Gamma$, $1\leq k\leq n$) is algebraic over the field $K(U_{\eta}(\overline{r}, \overline{s}))$.
In this case, since  $\gamma y_{k}\notin U(\overline{r}, \overline{s})$,  $\gamma y_{k}$ is equal to some term of the form $\gamma'u^{(1)}_{f_{j}}$ ($1\leq j\leq q$) where $\gamma'\in\Gamma$, $\gamma'\sim\gamma'u^{(1)}_{j}$,  $\ord_{i}u^{(i)}_{\gamma'f_{j}}\leq r_{i}$ for $i=2,\dots, p$, and $\ord_{l}v^{l}_{\gamma'f_{j}}\geq s_{l}$ for $l=1,\dots, p$.

Let us represent $f_{j}$ as a polynomial in $u^{(1)}_{f_{j}}$:
$$f_{j} = I_{d_{j}}^{(j)}(u^{(1)}_{f_{j}})^{d_{j}} + \dots + I_{1}^{(j)}u^{(1)}_{f_{j}} + I_{0}^{(j)}$$
where $I_{0}^{(j)}, I_{1}^{(j)},\dots I_{d_{j}}^{(j)}$ do not contain $u^{(1)}_{f_{j}}$ (therefore, all terms in these $\sigma^{\ast}$-polynomials are lower than $u^{(1)}_{f_{j}}$ with respect to $<_{1}$).
Since $f_{j}\in P$, $f_{j}(\eta) = 0$, that is,
\begin{equation}
I_{d_{j}}^{(j)}(\eta)(u^{(1)}_{f_{j}}(\eta))^{d_{j}} + \dots + I_{1}^{(j)}(\eta)u^{(1)}_{f_{j}}(\eta) + I_{0}^{(j)}(\eta) = 0.
\end{equation}
Note that $I_{d_{j}}^{(j)}(\eta)\neq 0$. Indeed, since $\rk I_{d_{j}}^{(j)} < \rk f_{j}$, the equality  $I_{d_{j}}^{(j)}(\eta) = 0$ would imply that $I_{d_{j}}^{(j)}\in P$. In this case, the family of all $f_{l}$ with $\rk f_{l} < \rk I_{d_{j}}^{(j)}$ and $I_{d_{j}}^{(j)}$ would form an $E$-autoreduced set in $P$ whose rank is lower than the rank of ${\mathcal{A}}$. This contradicts the fact that ${\mathcal{A}}$ is an  $E$-characteristic set of $P$. Similarly, $I_{\nu}^{(j)}\notin P$ for any $\nu = 0,\dots, d_{j}$ (and any $j=1,\dots, q$) and since $P$ is a $\sigma^{\ast}$-ideal, $\gamma(I_{\nu}^{(j)})\notin P$ for any $I_{\nu}^{(j)}$,  $\gamma\in\Gamma$. Therefore, if we apply $\gamma'$  to both sides of (10), the resulting equality will show that the element $\gamma'u^{(1)}_{f_{j}}(\eta) = \gamma\eta_{k}$ is algebraic over the field
$K(\{\tilde{\gamma}\eta_{l}\,|\,s_{i}\leq \ord_{i}\tilde{\gamma}\leq r_{i}\, (1\leq i\leq p), \tilde{\gamma} y_{l} <_{1}\gamma'u^{(1)}_{f_{j}}\})$. (Note that if $I = I_{\nu}^{(j)}$ for some $j\in\{1,\dots, q\}$ and
$\nu\in\{0,\dots, d_{j}\}$, then $\ord_{i}(\gamma'u_{I}^{(i)})\leq\ord_{i}u_{\gamma'I}^{(i)}\leq r_{i}$ ($2\leq i\leq p$) and $\ord_{k}(\gamma'v_{I}^{(k)})\geq\ord_{k}v_{\gamma'I}^{(k)}\leq s_{k}$ ($1\leq k\leq p$)\,).
Now, the induction on the well-ordered (with respect to $<_{1}$) set of terms $\Gamma Y$ completes the proof of the fact that the set $U_{\eta}(\overline{r}, \overline{s})$ is a transcendence basis of the field $K(W_{\eta}(\overline{r}, \overline{s})$ over $K$.

In order to evaluate the size of $U_{\eta}(\overline{r}, \overline{s})$ we are going to evaluate the sizes of the sets $U'_{\eta}(\overline{r}, \overline{s})$ and $U''_{\eta}(\overline{r}, \overline{s})$, that is, the sizes of the sets $U'(\overline{r}, \overline{s})$ and $U''(\overline{r}, \overline{s})$. For every $k=1,\dots, n$,
let $$A_{k} = \{(i_{1},\dots, i_{m})\in\mathbb{Z}^{m}\,\mid\,\alpha_{1}^{i_{1}}\dots\alpha_{m}^{i_{m}}y_{k}\,\,\, \text{is the $1$-leader of some element of}\,\,\,  \mathcal{A}\}.$$

Applying Theorem 2.8, we obtain that there exists a numerical polynomial $\omega_{k}(t_{1},\dots, t_{p})$ in $p$ variables with rational coefficients such that
$\omega_{k}(r_{1},\dots, r_{p}) = \Card W_{A_{k}}(r_{1},\dots, r_{p})$ for all sufficiently large $(r_{1},\dots, r_{p})\in\mathbb{N}^{p}$. It follows that
if we set $\psi_{\eta|K}(t_{1},\dots, t_{p}) = \D\sum_{k=1}^{n}\omega_{k}(t_{1},\dots, t_{p})$, then there exist $r^{(0)}_{i}, s_{i}^{(0)}, s^{(1)}_{i}\in\mathbb{N}$ ($1\leq i\leq p$) with $s_{i}^{(1)} < r^{(0)}_{i}-s_{i}^{(0)}$ such that for all $\overline{r}=(r_{1},\dots, r_{p}), \overline{s} = (s_{1},\dots, s_{p})\in\mathbb{N}^{p}$ with $r_{i}\geq r_{i}^{(0)}$, $s_{i}^{(1)}\leq s_{i}\leq r_{i}-s_{i}^{(0)}$, one has
\begin{equation}
\Card U_{\eta}(\overline{r}, \overline{s}) = \psi_{\eta|K}(r_{1},\dots, r_{p}) - \psi_{\eta|K}(s_{1}-1,\dots, s_{p}-1).
\end{equation}
Furthermore, $\deg\psi_{\eta|K}\leq m$, and $\deg\psi_{\eta|K} = m$ if and only if at least one of the sets $A_{k}$ ($1\leq k\leq n$) is empty.

In order to evaluate $\Card U''(\overline{r}, \overline{s})$, note that this set consists of all terms $\gamma u^{(1)}_{f_{j}}$  ($\gamma\in\Gamma$, $\gamma\sim u^{(1)}_{f_{j}}, 1\leq j\leq q$) such that
$s_{i}\leq\ord_{i}u^{(1)}_{\gamma f_{j}}\leq r_{i}$ and either $\ord_{1}v^{(1)}_{\gamma f_{j}} < s_{1}$ or there exists $k\in\{2,\dots, p\}$ such that $\ord_{k}u^{(k)}_{\gamma f_{j}} > r_{k}$ or there exists
$i\in\{2,\dots, p\}$ such that  $\ord_{i}v^{(i)}_{\gamma f_{j}} < s_{i}$ (``or'' is inclusive). It follows from Remarks 3.6 and 2.10 that if we fix $j$, the number of such terms $\gamma u^{(1)}_{f_{j}}$
satisfying the conditions $\ord_{i}v^{(i)}_{\gamma f_{j}} = \ord\gamma + b_{if_{j}} < s_{i}$, $\ord_{i}(\gamma u^{(1)}_{f_{j}}) = \ord_{i}\gamma + a_{1f_{j}}\geq s_{i}$ for $i\in\{k_{1},\dots, k_{d}\}\subseteq\{1,\dots, p\}$, $\ord_{i}(v^{(i)}_{\gamma f_{j}}) = \ord\gamma + b_{if_{j}}\geq s_{i}$ for
$i\in\{1,\dots, p\}$, $i\neq k_{\nu}$ ($1\leq \nu\leq d$) and  $\ord_{i}u^{(i)}_{\gamma f_{j}} = \ord\gamma + a_{if_{j}}\leq r_{i}$ for $i=1,\dots, p$ is equal to
$$\prod_{\substack{1\leq i\leq p,\\ i\neq k_{1},\dots, k_{d}}}\Bigg[\sum_{\mu=0}^{m_{i}}(-1)^{m_{i}-\mu}2^{\mu}{m_{i}\choose \mu}
\Bigg({{r_{i}- a_{if_{j}} +\mu}\choose \mu} - \hspace{3in}$$ $${{s_{i} - b_{if_{j}} +\mu-1}\choose \mu}\Bigg)\Bigg]\cdot \prod_{\nu=1}^{d}\Bigg[\sum_{\mu=0}^{m_{k_{\nu}}}(-1)^{m_{k_{\nu}}-\mu}2^{\mu}{m_{k_{\nu}}\choose \mu}\cdot \hspace{3in}$$
\begin{equation}
\Bigg({s_{k_{\nu}}- b_{k_{\nu}f_{j}}-1 + m_{k_{\nu}}\choose{m_{k_{\nu}}}} -{s_{k_{\nu}}- a_{1f_{j}}-1 + m_{k_{\nu}}\choose{m_{k_{\nu}}}}\Bigg)\Bigg]\hspace{1in}
\end{equation}
and a similar formula holds for the number of terms satisfying the conditions $\ord_{i}u^{(i)}_{\gamma f_{j}} >  r_{i}$ for $i\in\{l_{1},\dots, l_{e}\}\subseteq\{2,\dots, p\}$,
($\gamma\in\Gamma$, $\gamma\sim u^{(1)}_{f_{j}}$),  $\ord_{i}v^{(i)}_{\gamma f_{j}}\geq s_{i}$ for $i\in\{1,\dots, p\}$ and $\ord_{i}u^{(i)}_{\gamma f_{j}}\leq r_{i}$ for $i\neq l_{\nu}$ ($1\leq \nu\leq e$).

Applying the principle of inclusion and exclusion (taking into account terms that are multiples of more than one $1$-leaders), we obtain that $\Card U''(\overline{r}, \overline{s})$ is an alternating sum of polynomials in $r_{1},\dots, r_{p}, s_{1},\dots, s_{p}$ that are products of $k$ ($0\leq k\leq p$) terms of the form ${\D{r_{i}-a_{i}+m_{i}}\choose{\D m_{i}}} - {\D{s_{i}-b_{i} + m_{i}}\choose{\D m_{i}}}$ with $a_{i}, b_{i}\in\mathbb{N}$ ($1\leq i < p$) and
$p-k$ terms of the form either ${\D{s_{i}-c_{i}+m_{i}}\choose{\D m_{i}}} - {\D{s_{i}-d_{i} + m_{i}}\choose{\D m_{i}}}$ or ${\D{r_{i}-c_{i}+m_{i}}\choose{\D m_{i}}} - {\D{r_{i}-d_{i} + m_{i}}\choose{\D m_{i}}}$ with $c_{i}, d_{i}\in\mathbb{N}$, $c_{i} < d_{i}$. Since each such a polynomial has total degree at most $m-1$ and its degree with respect to $r_{i}$ or $s_{i}$ ($1\leq i\leq p$) does not exceed $m_{i}$, we obtain that
$\Card U''(\overline{r}, \overline{s}) = \lambda(r_{1},\dots, r_{p}, s_{1},\dots, s_{p})$ where $\lambda(t_{1},\dots, t_{2p})$ is a numerical polynomial in $2p$ variables such that $\deg\lambda < m$ and
$\deg_{t_{i}}\lambda\leq m_{i}$, $\deg_{t_{j}}\lambda\leq m_{j-p}$ for $i=1,\dots, p$, $j=p+1,\dots, 2p$.  It follows that the numerical polynomial
$$\Phi_{\eta\,|\,K}(t_{1},\dots, t_{2p}) = \psi_{\eta\,|,K}(t_{1},\dots, t_{p}) - \psi_{\eta\,|\,K}(t_{p+1}-1,\dots, t_{2p}-1) + \lambda(t_{1},\dots, t_{2p})$$ satisfies conditions of our theorem.
\end{proof}

\begin{definition}
The numerical polynomial $\Phi_{\eta\mid K}(t_{1},\dots, t_{2p})$ whose existence is established by Theorem 4.1 is called the {\em $2p$-variate $\sigma^{\ast}$-dimension polynomial} of the $\sigma^{\ast}$-field extension $L/K$ associated with the system of $\sigma^{\ast}$-generators $\eta$ and partition (6) of the set $\sigma$.
\end{definition}

The following theorem describes some invariants of a $2p$-variate $\sigma$-dimension polynomial of a finitely generated $\sigma^{\ast}$-field extension $L/K$ with partition (6) of $\sigma$, that is, characteristics of
the extension that do not depend on the set of $\sigma^{\ast}$-generators of $L$ over $K$.  In what follows we use the following notation. For any permutation
$(j_{1},\dots, j_{2p})$ of the set $\{1,\dots, 2p\}$, let $<_{j_{1},\dots, j_{2p}}$ denote the lexicographic order on $\mathbb{N}^{2p}$ such that $(k_{1},\dots, k_{2p})<_{j_{1},\dots,
j_{2p}} (l_{1},\dots, l_{2p})$ if and only if either $k_{j_{1}} < l_{j_{1}}$ or there exists $q\in\mathbb{N}$, $2\leq q\leq 2p$, such that $k_{j_{\nu}} = l_{j_{\nu}}$ for $\nu < q$ and
$k_{j_{q}} < l_{j_{q}}$.

\begin{theorem}
With the notation of Theorem {\em 4.1}, let $\Phi_{\eta\,|\,K}(t_{1},\dots, t_{2p})$ be the $2p$-variate $\sigma^{\ast}$-dimension polynomial of the $\sigma^{\ast}$-field extension
$L = K\langle \eta_{1},\dots,\eta_{n}\rangle^{\ast}$. Since the degrees of $\Phi_{\eta\,|\,K}$ with respect to $t_{i}$ and $t_{p+i}$ ($1\leq i\leq p$) do not exceed $m_{i}=\Card\sigma_{i}$ (see partition {\em (6)}), Theorem {\em 2.3} shows that this polynomial can be written as $$\Phi_{\eta\,|\,K}=\sum_{i_{1}=0}^{m_{1}}\dots \D\sum_{i_{p}=0}^{m_{p}}\sum_{i_{p+1}=0}^{m_{1}}\dots \D\sum_{i_{2p}=0}^{m_{p}}{a_{i_{1}\dots i_{2p}}}{t_{1}+i_{1}\choose i_{1}}\dots{t_{2p}+i_{2p}\choose i_{2p}}.$$
Let $E_{\eta} = \{(i_{1},\dots, i_{2p})\in\mathbb{N}^{2p}\,\mid\, 0\leq i_{k}, i_{p+k}\leq m_{k}$ ($k=1,\dots, p$) and $a_{i_{1}\dots i_{2p}}\neq 0\}$.
Then the total degree $d$ of $\Phi_{\eta\,|\,K}$ with respect to $t_{1},\dots, t_{p}$ and the coefficients of the terms of total degree $d$ in $\Phi_{\eta\,|\,K}$ do not depend on the choice of the set of $\sigma$-generators $\eta$. Furthermore, if $(\mu_{1},\dots, \mu_{p})$ is any permutation of $\{1,\dots, p\}$ and $(\nu_{1},\dots, \nu_{p})$ is any permutation of $\{p+1,\dots, 2p\}$, then the maximal element of $E_{\eta}$ with respect to the lexicographic order $<_{\mu_{1},\dots, \mu_{p}, \nu_{1},\dots, \nu_{p}}$ and the corresponding coefficient $a_{\mu_{1},\dots, \mu_{p}, \nu_{1},\dots, \nu_{p}}$ do not depend on the choice of a finite set of $\sigma$-generators of $L/K$ either. Finally, $a_{m_{1}\dots m_{p}0\dots 0} = a_{0\dots 0 m_{1}\dots m_{p}} = \sigma$-$\trdeg_{K}L$.
\end{theorem}

\begin{proof}
Suppose that $\zeta = \{\zeta_{1},\dots, \zeta_{l}\}$ is another set of $\sigma^{\ast}$-generators of $L/K$, that is, $L = K\langle \eta_{1},\dots,\eta_{n}\rangle^{\ast} = K\langle\zeta_{1},\dots,\zeta_{l}\rangle^{\ast}$. Let
$$\Phi_{\zeta\,|\,K}(t_{1},\dots, t_{2q}) = \D\sum_{i_{1}=0}^{m_{1}}\dots\sum_{i_{p}=0}^{m_{p}}\D\sum_{i_{p+1}=0}^{m_{1}}\dots\sum_{i_{2p}=0}^{m_{p}}b_{i_{1}\dots i_{2p}}{t_{1}+i_{1}\choose i_{1}}\dots{t_{2p}+i_{2p}\choose i_{2p}}$$
be the $2p$-variate dimension polynomial of the extension $L/K$ associated with the system of $\sigma^{\ast}$-generators $\zeta$.
Then there exist $h_{1},\dots, h_{p}\in\mathbb{N}$ such that $\eta_{i} \in K(\bigcup_{j=1}^{l}\Gamma(h_{1},\dots, h_{p})\zeta_{j})$ and $\zeta_{k}\in K(\bigcup_{j=1}^{n}\Gamma(h_{1},\dots, h_{p})\eta_{j})$ for any $i=1,\dots, n$ and $k=1,\dots, l$. (If $\Gamma'\subseteq \Gamma$, then $\Gamma'\zeta_{j}$ denotes the set $\{\gamma\zeta_{j}\,|\,\gamma\in\Gamma'\}$.) It follows that there exist
$r^{(0)}_{i}, s_{i}^{(0)}, s^{(1)}_{i}\in\mathbb{N}$ ($1\leq i\leq p$) with $s_{i}^{(1)} < r^{(0)}_{i}-s_{i}^{(0)}$ such that whenever $r_{i}\geq r_{i}^{(0)}$, $s_{i}^{(1)}\leq s_{i}\leq r_{i}-s_{i}^{(0)}$
($1\leq i\leq p$), one has
$$\Phi_{\eta\,|\,K}(r_{1},\dots, r_{2p})\leq\Phi_{\zeta\,|\,K}(r_{1}+h_{1},\dots, r_{p}+h_{p}, r_{p+1}-h_{1},\dots, r_{2p}-h_{p})$$ and
$$\Phi_{\zeta\,|\,K}(r_{1},\dots, r_{2p})\leq\Phi_{\zeta\,|\,K}(r_{1}+h_{1},\dots, r_{p}+h_{p}, r_{p+1}-h_{1},\dots, r_{2p}-h_{p}).$$

Now the statement of the theorem about the maximal elements of $E_{\eta}$ with respect to the lexicographic orders $<_{\mu_{1},\dots, \mu_{p}, \nu_{1},\dots, \nu_{p}}$ and the corresponding coefficients
follows from the fact that for any element $(k_{1},\dots, k_{2p})\in E_{\eta}'$, the term ${t_{1}+k_{1}\choose k_{1}}\dots {t_{2p}+k_{2p}\choose k_{2p}}$ appears in
$\Phi_{\eta\mid K}(t_{1},\dots, t_{2p})$ and $\Phi_{\zeta\mid K}(t_{1},\dots, t_{2p})$ with the same coefficient  $a_{k_{1}\dots k_{2p}}$. The equality of the coefficients of the corresponding terms of total
degree $d = \deg\Phi_{\eta\,|\,K} = \deg\Phi_{\zeta,|\,K}$ in $\Phi_{\eta,|\,K}$ and $\Phi_{\zeta\,|\,K}$ can be shown as in the proof of \cite[Theorem 3.3.21]{L4}.

In order to prove the last part of the theorem, note that the expression (12) and a similar expression corresponding to the conditions with
$\ord_{i}u^{(i)}_{\gamma f_{j}} >  r_{i}$ for $i\in\{l_{1},\dots, l_{e}\}\subseteq\{2,\dots, p\}$, ($\gamma\in\Gamma$, $\gamma\sim u^{(1)}_{f_{j}}$),  $\ord_{i}v^{(i)}_{\gamma f_{j}}\geq s_{i}$ for $i\in\{1,\dots, p\}$ and $\ord_{i}u^{(i)}_{\gamma f_{j}}\leq r_{i}$ for  $i\neq l_{\nu}$, $\nu =1,\dots, e$ (see the proof of Theorem 4.1) have the property that their total degrees with respect to $r_{1},\dots, r_{p}$ and $s_{1},\dots, s_{p}$ are less than $m$. It follows that the coefficients of the terms of total degree $m$ in $t_{1},\dots, t_{p}$ and terms of total degree $m$ in $t_{p+1},\dots, t_{2p}$  in the polynomial $\Phi_{\eta\,|\,K}$  are equal to the corresponding coefficients in the polynomials $\psi_{\eta\,|\,K}(t_{1},\dots, t_{p})$ and $\psi_{\eta\,|\,K}(t_{p+1},\dots, t_{2p})$, respectively (see (11)). Now, using the fact that if elements $\eta_{i_{1}},\dots,\eta_{i_{k}}$ ($i_{1},\dots, i_{k}\in\{1,\dots, n\}$) are $\sigma$-algebraically independent over $K$, then $\trdeg_{K}K((\{\gamma\eta_{i_{j}}\,|\,\gamma\in\Gamma(r_{1}, \dots, r_{p}; s_{1},\dots, s_{p}), 1\leq j\leq k\}) =\\
k\D\prod_{i=1}^{p}\left[\D\sum_{j=0}^{m_{i}}(-1)^{m_{i}-j}2^{j}{m_{i}\choose j}\left({{r_{i}+j}\choose j} - {{s_{i}+j-1}\choose j}\right)\right]$ for any $r_{i}, s_{i}\in\mathbb{N}$ with $s_{i}\leq r_{i}$ ($1\leq i\leq p$),
one can mimic the proof of \cite[Theorem 6.4.8]{KLMP} to obtain that $a_{m_{1}\dots m_{p}0\dots 0} = a_{0\dots 0 m_{1}\dots m_{p}} = \sigma$-$\trdeg_{K}L$.
\end{proof}

\begin{example}
Let $K$ be an inversive difference ($\sigma^{\ast}$-) field with a basic set $\sigma = \{\alpha_{1}, \alpha_{2}, \alpha_{3}\}$ considered together with its partition
$\sigma = \{\alpha_{1}\}\cup\{\alpha_{2}\}\cup\{\alpha_{3}\}$.  Let $L = K\langle \eta \rangle^{\ast}$ be a $\sigma^{\ast}$-field extension with the defining equation
\begin{equation}\alpha_{1}^{a}\eta + \alpha_{1}^{-a}\eta + \alpha_{2}^{b}\eta + \alpha_{3}^{c}\eta = 0
\end{equation}
where  $a, b, c\in\mathbb{N}$, $a > b > c > 0$. It means that the defining $\sigma^{\ast}$-ideal $P$ of the extension $L/K$ is a linear $\sigma^{\ast}$-ideal of the ring of $\sigma^{\ast}$-polynomials $K\{y\}^{\ast}$ generated by the linear $\sigma^{\ast}$-polynomial $f=\alpha_{1}^{a}y + \alpha_{1}^{-a}y + \alpha_{2}^{b}y + \alpha_{3}^{c}y$.

By Proposition 3.16, the $\sigma^{\ast}$-polynomials $f$ and $\alpha_{1}^{-1}f = \alpha_{1}^{-(a+1)}y + \alpha_{1}^{a-1}y + \alpha_{1}^{-1}\alpha_{2}^{b}y + \alpha_{1}^{-1}\alpha_{3}^{c}y$ form an $E$-characteristic set of $P$. Setting $\overline{r} = (r_{1}, r_{2}, r_{3})$, $\overline{s} = (s_{1}, s_{2}, s_{3})$ and using the notation of the proof of Theorem 4.1, we obtain (applying  Theorems 2.6 and 2.8) that for all sufficiently large
$(r_{1}, r_{2}, r_{3}, s_{1}, s_{2}, s_{3})\in\mathbb{N}^{6}$,

$\Card U'_{\eta}(\overline{r}, \overline{s}) = \phi_{\{(a, 0, 0), (-a-1, 0, 0)\}}(r_{1}, r_{2}, r_{3}, s_{1}, s_{2}, s_{3}) =$

$2a(2r_{2}-2s_{2}+2)(2r_{3}-2s_{3}+2)$

\smallskip

Furthermore, using the method of inclusion and exclusion (as it is indicated in the proof of Theorem 4.1), we get
$$\Card U''_{\eta}(\overline{r}, \overline{s}) = (2a+1)(2r_{2}-2s_{2} + 2)(2r_{3}-2s_{3} + 2) + 4b(r_{1}-s_{1} + 1)(2r_{3}-2s_{3} + 2) +\hspace{1in}$$
$$4c(r_{1}-s_{1} + 1)(2r_{2}-2s_{2} + 2) - 2b(2a+1)(2r_{3}-2s_{3} + 2) - 2c(2a+1)(2r_{2}-2s_{2} + 2) -\hspace{2in}$$
$$8bc(r_{1}-s_{1} + 1) + 8abc + 4bc.\hspace{4in}$$
Since the $6$-variate $\sigma^{\ast}$-dimension polynomial $\Phi_{\eta\,|\,K}$ expresses the number of elements of the set $U'_{\eta}(\overline{r}, \overline{s})\cup \Card U''_{\eta}(\overline{r}, \overline{s})$ for all sufficiently large values of its arguments, we obtain
$$\Phi_{\eta\,|\,K}(t_{1},\dots, t_{6}) = 8ct_{1}t_{2} + 8bt_{1}t_{3} - 8ct_{1}t_{5} - 8bt_{1}t_{6} + 4(4a+1)t_{2}t_{3} - 8ct_{2}t_{4} -\hspace{2in}$$
$$ 4(4a+1)t_{2}t_{6} - 8bt_{3}t_{4} - 4(4a+1)t_{3}t_{5} + 8ct_{4}t_{5} + 8bt_{4}t_{6} + 4(4a+1)t_{5}t_{6} +   \hspace{3in}$$
\begin{equation}
\text{the linear combination of monomials of total degree at most}\,\,\, 1.\hspace{5in}
\end{equation}
The univariate $\sigma^{\ast}$-dimension polynomial $\phi_{\eta\,|\,K}(t)$ (see theorem 2.1) is as follows (by \cite[Theorem 6.4.8]{KLMP}, it coincides with the dimension polynomial of the set
$A=\{(a, 0, 0), (-a-1, 0, 0)\}\subset\mathbb{Z}^{3}$, so it can be computed using Theorems 2.8 and 2.6 with $p=1$).
$$\phi_{\eta\,|\,K}(t) = 4at^{2} +  \,\,\,\text{the linear combination of monomials of degree at most}\,\,\, 1.\hspace{5in}$$
By Theorem 4.3, $\deg\,\Phi_{\eta\,|\,K} = 2$ and the coefficients of the terms $t_{i}t_{j}$ ($1\leq i, j\leq 6$) are invariants of the extension $L/K$, that is, they do not depend on the set of $\sigma^{\ast}$-generators of this extension. Therefore, the polynomial $\Phi_{\eta\,|\,K}(t_{1},\dots, t_{6})$ carries all three parameters $a$, $b$ and $c$ of the defining equation (13). At the same time, the univariate polynomial
$\phi_{\eta\,|\,K}(t)$ carries only the parameter $a$.
\end{example}

The fact that the $2p$-variate $\sigma^{\ast}$-dimension polynomial carry more invariants than its univariate counterpart can be applied to the equivalence problem for algebraic difference equations.
Suppose that we have two systems of algebraic difference ($\sigma$-) equations over a $\sigma^{\ast}$-field $K$  (i. e., equations of the form $f_{i}=0$ ($i\in I$) where all $f_{i}$ lie in some ring of $\sigma^{\ast}$-polynomials $K\{y_{1},\dots, y_{n}\}^{\ast}$) that are defining equations of finitely generated $\sigma^{\ast}$-field extensions $L/K$ and $L'/K$ (that is, the left-hand sides of the systems generate prime $\sigma^{\ast}$-ideals $P$ and $P'$ in the corresponding rings of $\sigma^{\ast}$-polynomials $R$ and $R'$ (possibly of different numbers of $\sigma^{\ast}$-generators) such that $L$ and $L'$ are $\sigma$-isomorphic to $\qf(R/P)$ and $\qf(R'/P')$, respectively). These systems are said to be {\em equivalent} if there is a $\sigma$-isomorphism between $L$ and $L'$ which is identity on $K$.
The  $2p$-variate $\sigma^{\ast}$-dimension polynomial introduced by Theorem 4.1 allows one to figure out that two systems of partial algebraic $\sigma$-equations are not equivalent in the case when the corresponding $\sigma^{\ast}$-field extensions have the same univariate $\sigma^{\ast}$-dimension polynomials.  As an example, consider the difference equations
\begin{equation} \alpha_{1}^{a}\eta + \alpha_{1}^{-a}\eta + \alpha_{2}^{b}\eta + \alpha_{3}^{c}\eta = 0
\end{equation}
and
\begin{equation} \alpha_{1}^{a}\eta + \alpha_{1}^{-a}\eta + \alpha_{2}^{d}\eta + \alpha_{3}^{e}\eta = 0
\end{equation}
where $a, b, c, d, e\in\mathbb{N}$, $a > b > c > 0$ and $a > d > e > 0$.

The invariants carried by the univariate $\sigma^{\ast}$-dimension polynomials associated with these equations (the equation (15) is considered in the last Example) are the same, the degree $1$ and $a$. At the same time, the 6-variate dimension polynomials for these equations carry invariants $a, b, c$, and $d, e, c$, respectively (these 6-variate dimension polynomials are of the form (14)). Thus, the difference equations (15) and (16) are not equivalent, even though the corresponding $\sigma^{\ast}$-field extensions have the same invariants carried by the univariate $\sigma^{\ast}$-dimension polynomials.

\section{Acknowledges}

This research was supported by the NSF grant CCF--2139462.

\end{document}